\documentclass[11pt]{amsart}
\usepackage{amsfonts}
\usepackage{amsmath}
\usepackage{amsthm}
\usepackage{labelfig}
\usepackage{epsfig}
\usepackage{epstopdf}

\usepackage{url}
\usepackage[all]{xy}
\usepackage{latexsym}
\usepackage{amssymb}
\usepackage[cp850]{inputenc}
\usepackage{amsfonts}
\usepackage{graphicx}
\usepackage{dsfont}

\usepackage[usenames]{color}

\newtheorem{sat}{Theorem}[section]		
\newtheorem{lem}[sat]{Lemma}
\newtheorem{kor}[sat]{Corollary}

\newtheorem*{defi*}{Definition}			
\newtheorem*{bei*}{Example}
\newtheorem*{sat*}{Theorem}				
\newtheorem*{kor*}{Corollary}
\newtheorem*{rmk*}{Remark}				
	
\newtheorem*{quest*}{Question}

\let\ssection=\section
\renewcommand{\section}{\setcounter{equation}{0}\ssection}

\newtheorem*{namedtheorem}{\theoremname}
\newcommand{\theoremname}{testing}
\newenvironment{named}[1]{\renewcommand{\theoremname}{#1}\begin{namedtheorem}}{\end{namedtheorem}}

\theoremstyle{remark}
\newtheorem*{bem}{Remark}

\newtheorem*{namedtheoremr}{\theoremnamer}
\newcommand{\theoremnamer}{testing}

\newcommand{\BR}{\mathbb R}			
\newcommand{\BN}{\mathbb N}			\newcommand{\BQ}{\mathbb Q}

\newcommand{\CC}{\mathcal C}

		\newcommand{\CL}{\mathcal L}
\newcommand{\CM}{\mathcal M}

\newcommand{\CS}{\mathcal S}		
\newcommand{\CU}{\mathcal U}

\newcommand{\actson}{\curvearrowright}
\newcommand{\D}{\partial}

\DeclareMathOperator{\Out}{Out}		%	Aeussere Automorphismen einer Gruppe
\DeclareMathOperator{\Id}{Id}		%	Identit\"at
\DeclareMathOperator{\Isom}{Isom}	%	Isometrien einer Mf
\DeclareMathOperator{\Map}{Map}

\newcommand{\comment}[1]{}

\DeclareMathOperator{\Homeo}{Homeo}
\DeclareMathOperator{\Aut}{Aut}

\DeclareMathOperator{\Thu}{Thu}

\DeclareMathOperator{\Inn}{Inn}

\newcommand{\fsubd}{\mathrel{{\scriptstyle\searrow}\kern-1ex^d\kern0.5ex}}
\newcommand{\bsubd}{\mathrel{{\scriptstyle\swarrow}\kern-1.6ex^d\kern0.8ex}}

\begin{document}

\title[]{Counting curves, and the stable length of currents}
\author{Viveka Erlandsson}
\address{Aalto Science Institute, Aalto University}
\email{viveca.erlandsson@aalto.fi}
\author{Hugo Parlier}
\address{Department of Mathematics, University of Fribourg}
\email{hugo.parlier@unifr.ch}
\author{Juan Souto}
\address{IRMAR, Universit\'e de Rennes 1}
\email{juan.souto@univ-rennes1.fr}
\thanks{The first two authors acknowledge support from Swiss National Science Foundation grant number PP00P2\textunderscore 153024. Moreover, this material is based upon work supported by the National Science Foundation under Grant No. DMS-1440140 while the third author was in residence at the Mathematical Sciences Research Institute in Berkeley, California, during the Fall 2016 semester. Finally, the authors acknowledge support from U.S. National Science Foundation grants DMS 1107452, 1107263, 1107367 RNMS: Geometric structures And Representation varieties (the GEAR Network).}

\begin{abstract}
Let $\gamma_0$ be a curve on a surface $\Sigma$ of genus $g$ and with $r$ boundary components and let $\pi_1(\Sigma)\actson X$ be a discrete and cocompact action on some metric space. We study the asymptotic behavior of the number of curves $\gamma$ of type $\gamma_0$ with translation length at most $L$ on $X$. For example, as an application, we derive that for any finite generating set $S$, of $\pi_1(\Sigma)$ the limit 
$$\lim_{L\to\infty}\frac 1{L^{6g-6+2r}}\{\gamma\text{ of type }\gamma_0\text{ with }S\text{-translation length}\le L\}$$
exists and is positive. The main new technical tool is that the function which associates to each curve its stable length with respect to the action on $X$ extends to a (unique) continuous and homogenous function on the space of currents. We prove that this is indeed the case for any action of a torsion free hyperbolic group.
\end{abstract}
\maketitle

\section{}

Let $\Sigma$ be a compact surface, possibly with non-empty boundary, with $\chi(\Sigma)<0$, and other than a three holed sphere. We denote by 
$$\Map(\Sigma)=\Homeo(\Sigma)/\Homeo_0(\Sigma)$$ 
the (full) mapping class group and recall that it can be identified with a subgroup of the group $\Out(\pi_1(\Sigma))$ of exterior automorphisms of the fundamental group of the surface $\Sigma$. If $\Sigma$ is closed, then one has in fact that $\Map(\Sigma)=\Out(\pi_1(\Sigma))$.

\begin{bem}
The reader might have a preference for the orientation preserving mapping class group or, in the presence of boundary, might have a special spot in her or his heart for the pure mapping class group or for the group of isotopy classes of homeomorphisms which fix the boundary pointwise. All the results here hold if we replace the full mapping class group $\Map(\Sigma)$ by any of these groups or by any of their finite index subgroups.
\end{bem}

Free homotopy classes of curves in $\Sigma$ are naturally identified with conjugacy classes in the fundamental group. We will pass freely from one point of view to the other and, unless we want to stress a given point, refer to them simply as {\em curves}, and we say that a curve has a certain property if the individual representatives do. A curve is {\em essential} if it (its representatives) are neither homotopically trivial nor homotopic to a boundary component. The mapping class group acts on the set $\CS=\CS(\Sigma)$ of all essential curves and we say that two elements in $\CS$ in the same $\Map(\Sigma)$-orbit are {\em of the same type}. Given an essential curve $\gamma_0$, let $\CS_{\gamma_0}=\Map(\Sigma)\cdot \gamma_0$ be the set of all curves of type $\gamma_0$.

In this note we are interested in counting the number of elements in $\CS_{\gamma_0}$ which have translation length at most $L$ with respect to some action on a hyperbolic space. To be more precise, suppose that $\pi_1(\Sigma)\actson X$ is an action by isometries on some geodesic metric space $X=(X,d_X)$. Given a curve $\gamma\in\CS$, consider it as a conjugacy class in $\pi_1(\Sigma)$, choose a representative (which we still refer to as $\gamma$) and set
$$\ell_X(\gamma)=\inf_{x\in X}d_X(x,\gamma(x)).$$
The quantity $\ell_X(\gamma)$, which we call the {\em length of $\gamma$ with respect to the action $\pi_1(\Sigma)\actson X$}, does not depend on the chosen representative of the conjugacy class $\gamma$.

The following is our main theorem:

\begin{sat}\label{sat1}
Let $\Sigma$ be a compact surface with genus $g$ and $r$ boundary components and assume that $3g+r> 3$. Further let $\pi_1(\Sigma)\actson X$ be a discrete and cocompact isometric action on a geodesic metric space $X$, $\gamma_0$ an essential curve in $\Sigma$, and $\CS_{\gamma_0}=\Map(\Sigma)\cdot\gamma_0$ the set of all curves of type $\gamma_0$. Then the limit
$$\lim_{L\to\infty}\frac 1{L^{6g-6+2r}}\vert\{\gamma\in\CS_{\gamma_0}\,\vert\,\ell_X(\gamma)\le L\}\vert$$
exists and is positive. 
\end{sat}

In addition to the existence of the limit in Theorem \ref{sat1}, we will end up knowing something about its value: it decomposes as the product of a constant which depends only on the curve $\gamma_0$ and another which depends only on the action. In particular we will get:

\begin{kor}\label{kor-rational}
Let $\Sigma$ be a compact surface with genus $g$ and $r$ boundary components and assume that $3g+r> 3$. For any curve $\eta$ in $\Sigma$ there is $n_{\eta}\in\BQ$ such that
$$\lim_{L\to\infty}\frac{\vert\{\gamma\in\CS_{\eta}\,\vert\,\ell_X(\gamma)\le L\}\vert}{\vert\{\gamma\in\CS_{\eta'}\,\vert\,\ell_X(\gamma)\le L\}\vert}=\frac{n_\eta}{n_{\eta'}}$$
for any two essential curves $\eta,\eta'$ in $\Sigma$ and any discrete and cocompact action $\pi_1(\Sigma)\actson X$.
\end{kor}

Before sketching the proof of Theorem \ref{sat1} we present some applications. Endow the interior $\Sigma^0=\Sigma\setminus\D\Sigma$ of $\Sigma$ with some complete Riemannian metric $\rho$ and let $\ell_\rho(\gamma)$ be the $\rho$-length of a shortest curve in $\Sigma^0$ freely homotopic to $\gamma$. Note that $\ell_\rho(\gamma)$ agrees with the translation length of $\gamma$ with respect to the action by deck-transformations of $\pi_1(\Sigma)$ on the universal cover of $(\Sigma_0,\rho)$. We will show that Theorem \ref{sat1} implies:

\begin{kor}\label{kor-riemannian}
With $\Sigma$ as in Theorem \ref{sat1}, let $\rho$ be a complete Riemannian metric on $\Sigma^{0}=\Sigma\setminus\D\Sigma$. Then for every essential curve $\gamma_0$ the limit
$$\lim_{L\to\infty}\frac 1{L^{6g+2r-6}}\vert\{\gamma\in\CS_{\gamma_0}\,\vert\,\ell_\rho(\gamma)\le L\}\vert$$
exists and is positive.
\end{kor}

Corollary \ref{kor-riemannian} is due for hyperbolic metrics to Mirzakhani \cite{Maryam2}, generalizing her work on the growth of simple closed geodesics \cite{Maryam1}. For non-positively curved metrics, it is due to Erlandsson-Souto \cite{VJ}. In fact, the results of these papers are instrumental in the proof of Theorem \ref{sat1}. A natural example for which Corollary \ref{kor-riemannian} was previously unknown, even for simple curves, is when $\rho$ is given by the induced metric on a smooth embedding of $\Sigma$ in $\BR^3$. 

In a different direction, we can also consider the action of $\pi_1(\Sigma)$ on the Cayley graph $C(\pi_1(\Sigma),S)$ of $\pi_1(\Sigma)$ with respect to some finite symmetric set of generators $S\subset\pi_1(\Sigma)$. In this case the translation length of $\gamma$ with respect to the action $\pi_1(\Sigma)\actson C(\pi_1(\Sigma),S)$ is the minimal word length with respect to $S$ over all elements in $\pi_1(S)$ representing the conjugacy class $\gamma$. Since the action $\pi_1(\Sigma)\actson C(\pi_1(\Sigma),S)$ is discrete and cocompact, we get the following directly from Theorem \ref{sat1}:

\begin{kor}\label{kor-wordlength}
With $\Sigma$ as in Theorem \ref{sat1}, let $S\subset\pi_1(\Sigma)$ be a finite symmetric set of generators. Then, for every essential curve $\gamma_0$, the limit
$$\lim_{L\to\infty}\frac 1{L^{6g+2r-6}}\vert\{\gamma\in\CS_{\gamma_0}\,\vert\,\ell_S(\gamma)\le L\}\vert$$
exists and is positive. Here $\ell_S(\gamma)$ is the minimal word length with respect to $S$ over all elements in $\pi_1(\Sigma)$ representing the conjugacy class $\gamma$.\qed
\end{kor}

For so-called {\em simple} generating sets Corollary \ref{kor-wordlength} is due to the first author of this note \cite{Viveka}. 

\begin{bem}
The combined statements of Corollary \ref{kor-rational} and Corollary \ref{kor-wordlength} can be thought of as giving an answer to the generalization of Conjecture 1 in \cite{Moira} to all surfaces of negative Euler-characteristic and to arbitrary generating sets. It would be interesting to investigate if our methods can be of any use to study the other problems suggested in that paper.
\end{bem}

We discuss now the idea of the proof of Theorem \ref{sat1}. Unsurprisingly, its proof has common components with the proof of Corollary \ref{kor-riemannian} for non-positively curved metrics \cite{VJ} and that of Corollary \ref{kor-wordlength} for simple generating sets \cite{Viveka}. More concretely, we will use that certain limits of measures on the space of currents $\CC(\pi_1(\Sigma))$ exist. However, there is a key ingredient in the proofs of the two mentioned results which is lacking in our setting, namely that there are currents $\lambda_\rho,\lambda_S\in\CC(\pi_1(\Sigma))$ with $\ell_\rho(\gamma)=\iota(\lambda_\rho,\gamma)$ and $\ell_S(\gamma)=\iota(\lambda_S,\gamma)$ for all $\gamma$. Analyzing how the existence of these currents was used in \cite{VJ} and \cite{Viveka}, one sees that they essentially serve to prove that the corresponding length function, say $\ell_S(\cdot)$, admits a continuous and homogenous extension to the space of currents: set $\ell_S(\mu)=\iota(\lambda_S,\mu)$. Lacking such a current, we need to construct the extension in some other way. 

However, it is not always possible to extend the length function to a continuous homogenous function on the space of currents. As an example, consider the free group on two generators $a$ and $b$, and the word length $\ell_S(\cdot)$ with respect to the generating set $S=\{a,b,a^5\}$ and their inverses. Note that, considered as currents, the sequence $\left(\frac{1}{5n}a^{5n}b\right)$ converges to the current $a$ as $n\to\infty$. Hence, if such a continuous extension existed, we would have $\ell_S(\frac{1}{5n}a^{5n}b)\to\ell_S(a)=1$ as $n\to\infty$. On the other hand, using the homogeneity of the extension, 
$$\ell_S\left(\frac{1}{5n}a^{5n}b\right) = \frac{1}{5n}\ell_S\left(a^{5n}b\right) = \frac{n+1}{5n} \to\frac{1}{5}$$
as $n\to\infty$, a contradiction. Hence no such extension exists in general. 

We by-pass this problem by considering {\em stable lengths} instead. With the same notation, we define the stable length to be 
$$\Vert \gamma\Vert_X=\lim_{n\to\infty}\frac 1n\inf_{x\in X}d_X(x,\gamma^n(x))$$
for some (any) element $\gamma\in\pi_1(\Sigma)$ in the conjugacy class $\gamma$. The following result, which we prove for general hyperbolic groups in the hope that somebody will find use for it, shows that the stable length extends to the space of currents:

\begin{sat}\label{sat-length}
Let $\Gamma\actson X$ be a discrete and cocompact isometric action of a torsion free Gromov hyperbolic group  on a geodesic metric space and let $\CC^+(\Gamma)$ be the space of oriented currents of $\Gamma$. There is a unique continuous function
$$\Vert\cdot\Vert_X:\CC^+(\Gamma)\to\BR_+$$
with $\Vert\gamma\Vert_X=\lim_{n\to\infty}\frac 1n\inf_{x\in X}d_X(x,\gamma^n(x))$ for every non-trivial $\gamma\in\Gamma$. Moreover, the function $\Vert\cdot\Vert_X$ is flip-invariant.
\end{sat}

\begin{bem}
It is perhaps worth noticing that the stable length function $\Vert\cdot\Vert_X$ is linear, meaning that
$$\Vert t\cdot\lambda+s\cdot\mu\Vert_X=t\cdot\Vert\lambda\Vert_X+s\cdot\Vert\mu\Vert_X$$
for all $\lambda,\mu\in\CC^+(\Gamma)$ and $s,t\in\BR_+$. Recall that $\CC^+(\Gamma)$, being a space of measures, is a cone in a linear space. 
\end{bem}

Uniqueness of the function $\Vert\cdot\Vert_X$ follows from the density of the currents of the form $t\cdot\gamma$ with $t\in\BR_+$ and with $\gamma\in\Gamma$ of infinite order \cite{Bonahon-currents-group}. Note that Theorem \ref{sat-length} was proved in \cite{Bonahon-currents-group} for spaces $X$ with the property that any two points in the boundary at infinity determine a unique geodesic between them---this condition is, in general, neither satisfied for lifts of Riemannian metrics, nor for Cayley graphs. 

In any event, armed with Theorem \ref{sat-length}, we can follow the strategy from \cite{VJ} and \cite{Viveka} to prove a version of Theorem \ref{sat1} replacing length by stable length:

\begin{sat}\label{sat2}
Let $\Sigma$ be a compact surface with genus $g$ and $r$ boundary components and assume that $3g+r > 3$. Let also $\pi_1(\Sigma)\actson X$ be a discrete and cocompact isometric action on a geodesic metric space $X$, $\gamma_0$ an essential curve in $\Sigma$, and $\CS_{\gamma_0}=\Map(\Sigma)\cdot\gamma_0$ the set of all curves of type $\gamma_0$. Then the limit
$$\lim_{L\to\infty}\frac 1{L^{6g-6+2r}}\vert\{\gamma\in\CS_{\gamma_0}\text{ with }\Vert\gamma\Vert_X\le L\}\vert$$
exists and is positive. 
\end{sat}

We will then deduce Theorem \ref{sat1} from Theorem \ref{sat2}. In fact, we will get that the limits in both theorems agree.
%, and moreover that this common limit is actually equal to 
%$$C_{\gamma_0}\cdot\mu_{Thu}\left(\{\lambda\in\CM\CL\text{ with }\Vert\lambda\Vert_X\leq1\}\right)$$
%where $C_{\gamma_0}>0$ is a constant depending only on the type $\gamma_0$ and $\mu_{Thu}$ is the Thurston measure on the space of measured laminations $\CM\CL$.  
%$$\lim_{L\to\infty}\frac{\vert\{\gamma\in\CS_{\gamma_0}\vert\ell_X(\gamma)\le L\}\vert}{L^{6g-6+2r}}=\lim_{L\to\infty}\frac{\vert\{\gamma\in\CS_{\gamma_0}\text{ with }\Vert\gamma\Vert_X\le L\}\vert}{L^{6g-6+2r}}$$
\medskip

The paper is organized as follows. In section \ref{sec-hyperbolic} we recall general facts about hyperbolic spaces and hyperbolic groups. Then, in section \ref{sec-currents}, we remind the reader of what the geodesic flow of a hyperbolic group is and recall facts about currents. We suggest that experts, other than the referee, just breeze through these two sections because, other than serving to fix the terminology and notation, they only contain results there are either well-known or almost immediate. In section \ref{sec-meat} we prove Theorem \ref{sat-length} and finally in section \ref{sec-counting} we prove Theorem \ref{sat1}, Theorem \ref{sat2} and the corollaries.
\medskip

\noindent {\bf Acknowledgements.} The authors would like to thank Chris Leininger for some very interesting conversations on this and other topics. 

\section{}\label{sec-hyperbolic}
We recall a few facts about hyperbolic spaces and groups. See \cite{Gromov,GdlH} and \cite{BH} for definitions and details.

\subsection{Metric spaces and quasi-isometries}
In this paper, we assume all metric spaces $(X,d_X)$ to be {\em geodesic}, meaning that the distance $d_X$ is inner and that any two points $x,x'\in X$ are joined by a segment $[x,x']$ whose length agrees with the distance $d_X(x,x')$ between the points. Given two such metric spaces $X=(X,d_X)$ and $Y=(Y,d_Y)$, recall that a map 
$$\Phi:X\to Y$$
is an {\em $(L,A)$-quasi-isometric embedding} for some $L\ge 1$ and $A\ge 0$ if we have
$$\frac 1L\cdot d_X(x,x')-A\le d_Y(\Phi(x),\Phi(x'))\le L\cdot d_X(x,x')+A$$
for all $x,x'\in X$. A map is a quasi-isometric embedding if it is an $(L,A)$-quasi-isometric embedding for some choice of $L$ and $A$. Two quasi-isometric embeddings $\Phi,\Psi:X\to Y$ are {\em equivalent} if their images are at bounded distance, meaning that there is some $C$ with $d_Y(\Phi(x),\Psi(x))\le C$ for all $x\in X$. A quasi-isometric embedding $\Phi:X\to Y$ is a {\em quasi-isometry} if there is a further quasi-isometric embedding $\Psi:Y\to X$ such that $\Psi\circ\Phi$ and $\Phi\circ\Psi$ are equivalent to the respective identity maps $\Id_X$ and $\Id_Y$.

In our setting, quasi-isometries will appear via the well-known Milnor-Schwarz theorem:

\begin{named}{Milnor-Schwarz theorem}
Let $X,Y$ be geodesic metric spaces and $\Gamma$ a group. Let also $\Gamma\actson X$ and $\Gamma\actson Y$ be discrete and cocompact isometric actions. Then there is a $\Gamma$-equivariant quasi-isometry 
$$\Phi:X\to Y.$$
Moreover, any two such quasi-isometries are equivalent to each other.
\end{named}

As an example, let $\Gamma$ be a finitely generated group, let $S$ and $S'$ be finite symmetric generating sets, and let $C(\Gamma,S)$ and $C(\Gamma,S')$ be the associated Cayley graphs endowed with the inner metric with respect to which each edge has unit length. Then $\Gamma$ acts by isometries, discretely and cocompactly on both $C(\Gamma,S)$ and $C(\Gamma,S')$. In particular, it follows from the Milnor-Schwarz theorem that $C(\Gamma,S)$ and $C(\Gamma,S')$ are quasi-isometric.

\subsection{Hyperbolic spaces and groups}
A {\em triangle} in a metric space consists just of three points $x,y,z\in X$ and three geodesics $[x,y],[y,z]$ and $[z,x]$, one between any two pair of points. A metric space $X$ is {\em $\delta$-hyperbolic} for some $\delta\ge 0$ if we have
$$\max_{t\in[z,x]}\min_{s\in[x,y]\cup[y,z]}d_X(t,s)\le\delta$$
for every triangle with vertices $x,y,z\in X$ and with sides $[x,y],[y,z]$ and $[z,x]$. A metric space is {\em hyperbolic} if it is $\delta$-hyperbolic for some $\delta$. 

From a geometric point of view, the for us most important feature of hyperbolic spaces will be the {\em stability of quasi-geodesics}. Recall that an {\em $(L,A)$-quasi-geodesic} is nothing other than an $(L,A)$-quasi-geodesic embedding
$$\Phi:I\to X$$
of some subinterval of $\BR$. Note that $I$ can be open, closed, half-open, possibly a point, and  possibly the whole line.

\begin{named}{Stability of quasi-geodesics}
For all $L\ge 1$, $A\ge 0$ and $\delta\ge 0$ there is $D$ such that for any hyperbolic space $X$, for any finite interval $[a,b]\subset\BR$ and for any two $(L,A)$-quasi-geodesics
$$\Phi,\Psi:[a,b]\to X$$
with $\Phi(a)=\Psi(a)$ and $\Phi(b)=\Psi(b)$ we have
$$\sup_{s\in[a,b]}\inf_{t\in[a,b]}d_X(\Phi(s),\Psi(t))\le D.
$$
\end{named}

\begin{figure}[h]
%\ShowGrid
{
\leavevmode \SetLabels
\L(.46*.57) $\Phi$\\%
\L(.38*.28) $\Psi$\\
%\L(.40*.97) $\Psi$\\
\endSetLabels
\begin{center}
\AffixLabels{\centerline{\epsfig{file =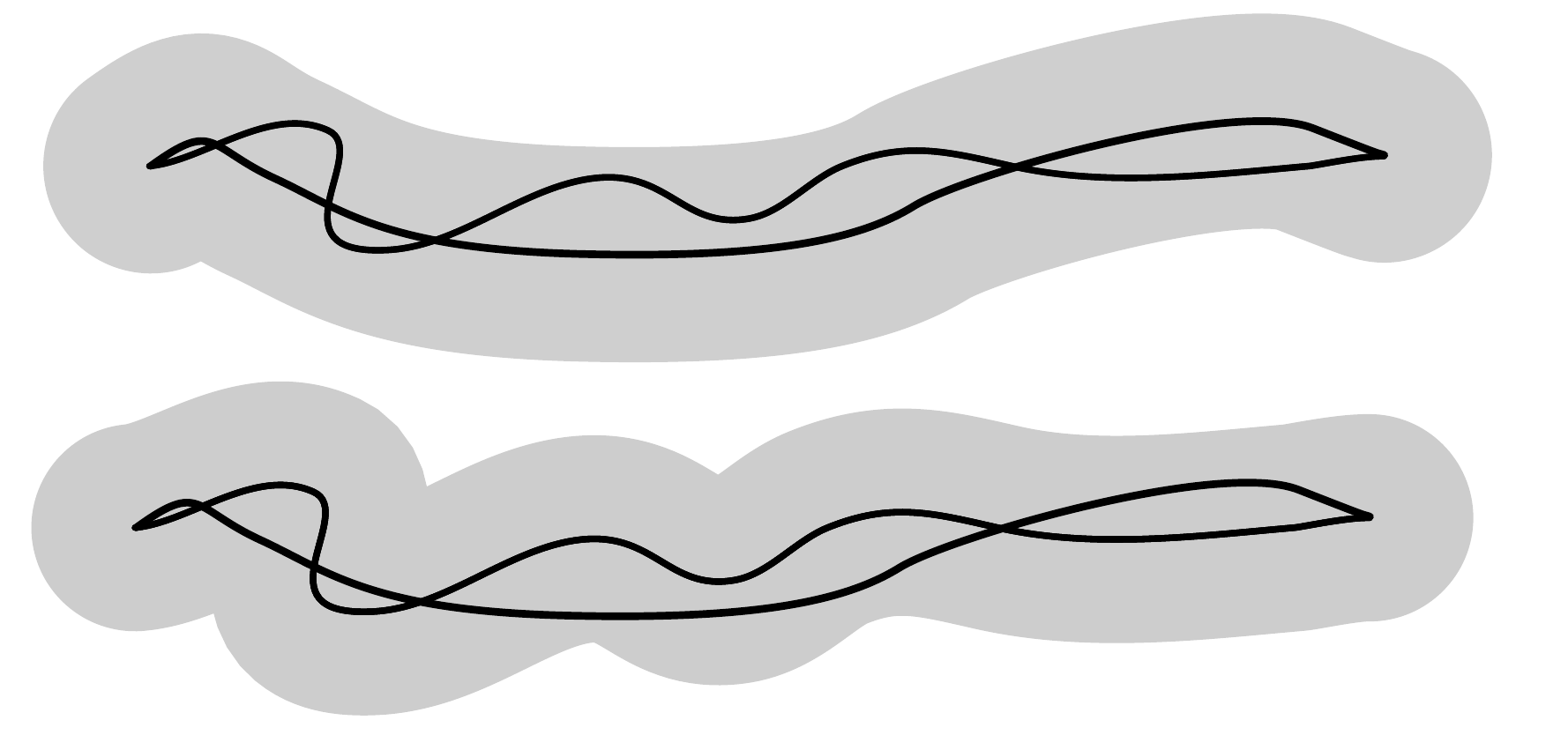,width=8.5cm,angle=0}}}
\vspace{-24pt}
\end{center}
}
\caption{Stability of quasi-geodesics: Any two $(L,A)$-quasi-geodesics are in the $D$-neighborhoods of each other.} \label{fig:stab}
\end{figure}

A more humane version of this last statement, illustrated in Figure \ref{fig:stab}, is the assertion that {\em quasi-geodesics in hyperbolic metric spaces fellow-travel}, but we wanted to stress that the constants only depend on the hyperbolicity and quasi-geodesic constants.

In general, determining if a given map is a quasi-geodesic might be very hard, but in the setting of hyperbolic metric spaces it is, once correctly formulated, a local problem. In fact, we have the following well-known fact:
% (see, for example, Bridson-Haefliger, Theorem 1.13):

\begin{named}{Local geodesics are quasi-geodesics}
For every $\delta>0$ there exist $D>0, L\geq 1, A\geq 0$ such that if $X$ is any $\delta$-hyperbolic space, $[a,b]\subset\mathbb{R}$, and $\Phi:[a,b]\to X$ is a $D$-local geodesic, then $\Phi$ is an $(L,A)$-quasi-geodesic. 
\end{named}

Here, a map $\Phi:[a,b]\to X$ is a {\em $D$-local geodesic} if its restriction $\Phi\vert_{[t_1,t_2]}$ to any segment $[t_1,t_2]\subset[a,b]$ of length $\vert t_1-t_2\vert<D$ is actually geodesic. In fact, the above also holds true for local quasi-geodesics, but it involves a few more quantifiers and here we will only need the stated version. 

Another useful observation is that for all $\delta$ there is some constant $D$ such that if $X$ is $\delta$-hyperbolic, then the space obtained by coning off all balls of radius $D$ is contractible. This is the key argument to ensure the contractibility of the so-called Rips complex. Anyways, if $\Gamma$ acts on $X$ by isometries, it also acts on this coned-off space by isometries. In fact, the metric on the cones can be chosen so that $X$ is totally convex in the coned-off space. Altogether we get the following fact, well-known to experts:

\begin{lem}\label{lem-cone off}
Every $\delta$-hyperbolic space $X$ is a totally convex subset of a contractible metric space $\hat X$ on which $\Isom(X)$ acts. Moreover, the inclusion $X\hookrightarrow\hat X$ is an $\Isom(X)$-equivariant quasi-isometry.\qed
\end{lem}

Finally, returning to the stability of quasi-geodesics, one of its key consequences is that any space quasi-isometric to a hyperbolic space is itself hyperbolic. This applies to the space $\hat X$ in Lemma \ref{lem-cone off}, but it also implies that the hyperbolicity or not of the Cayley graph $C(\Gamma,S)$ of a group $\Gamma$, with respect to a finite generating set $S$, does not depend on the latter. Accordingly, if $C(\Gamma,S)$ is hyperbolic for some $S$, then we say that $\Gamma$ itself is a {\em hyperbolic group}.

\subsection{Boundary at infinity}
Two geodesic or quasi-geodesic rays in a hyperbolic space $X$ are {\em asymptotic} if they fellow travel in the sense of the theorem on stability of quasi-geodesics. Being asymptotic is an equivalence relation and the boundary $\D X$ at infinity of $X$ is the set of all equivalence classes of geodesic rays. It has a natural topology, derived for example from the compact-open topology. %Note that it follows from the stability of quasi-geodesics that one obtains the same set $\D X$ if one considers equivalence classes of quasi-geodesic rays instead.

To every geodesic $\BR\to X$ we can associate two points in the boundary: one corresponding to the positive ray and the other to the negative one. It is in fact easy to see that they are different, which means that the pair consisting of the two end-points of a geodesic belongs to the so-called {\em double boundary} of $X$:
$$\D^2X=\D X\times\D X\setminus\Delta.$$
Here $\Delta$ is the diagonal. 

Finally, note that it follows from the stability of quasi-geodesics that any quasi-isometry $\Phi:X\to Y$ between two hyperbolic spaces induces maps
$$\D\Phi:\D X\to\D Y\text{ and }\D^2\Phi:\D^2X\to\D^2Y$$
between the associated boundaries and double boundaries. These maps are homeomorphisms. This implies in particular that, up to homeomorphism, the boundary $\D C(\Gamma,S)$ of the Cayley graph of a hyperbolic group with respect to a finite generating set does not depend on the generating set in question. Accordingly, we set 
$$\D\Gamma:=\D C(\Gamma,S)$$
and speak about the boundary of the group. In the same way we consider the double boundary $\D^2\Gamma:=\D\Gamma\times\D\Gamma\setminus\Delta$ of $\Gamma$. 

Note finally that every automorphism $\phi\in\Aut(\Gamma)$ induces a homeomorphism $\phi_*:\D\Gamma\to\D\Gamma$, equivariant in the following sense: $\phi_*(\gamma\cdot\theta)=\phi(\gamma)\cdot\phi_*(\theta)$. It goes without saying that the same remains true if we replace $\D\Gamma$ by the double boundary $\D^2\Gamma$.

\subsection{Useful facts}
We discuss now two technical facts needed later on. The first basically asserts the following: if we have a very long quasi-geodesic and we want to estimate the distance between its endpoints within a few percentage points, then it suffices to sum up the distances between a collection of consecutive intermediate points satisfying only the condition that they are miles away from each other. More precisely:

\begin{lem}\label{lem-almost right}
For every $\delta\ge 0$, $L\ge 1$, $A\ge 0$ and $\epsilon>0$ there is $R$ such that for any $\delta$-hyperbolic space $X$ and any $(L,A)$-quasi-geodesic $\Phi:[0,T]\to X$ we have
$$\sum_{i=1}^sd_X(\Phi(t_{i-1}),\Phi(t_i))\ge d_X(\Phi(0),\Phi(T))\ge (1-\epsilon)\sum_{i=1}^sd_X(\Phi(t_{i-1}),\Phi(t_i))$$
for any collection of points $0=t_0<t_1<\dots<t_s=T$ with $t_i-t_{i-1}\ge R$ for all $i=1,\dots,s$.
\end{lem}

\begin{proof}
Let $\Psi:[0,T]\to X$ be a geodesic with endpoints $\Psi(0)=\Phi(0)$ and $\Psi(T)=\Phi(T)$. Because of the stability of quasi-geodesics, there is some $D=D(L,A)$ such that for each $i=0,\dots,s$ there exists $\bar t_i\in[0,T]$ with 
$$d(\Phi(t_i),\Psi(\bar t_i))\le D.$$
Note that we can choose $\bar t_0=t_0=0$ and $\bar t_s=t_s=T$. Note also that as long as $R$, and thus the gaps $t_i-t_{i-1}\ge R$, is chosen big enough we have that
$$0=\bar t_0<\bar t_1<\dots<\bar t_s=T.$$

\begin{figure}[h]
%\ShowGrid
{
\leavevmode \SetLabels
\L(-.09*.43) $\Phi(0)=\Psi(0)$\\%
\L(.17*.43) $\Psi(t_1)$\\%
\L(.15*1.01) $\Phi(\bar{t}_1)$\\
\L(.335*.43) $\Psi(t_2)$\\%
\L(.33*-.07) $\Phi(\bar{t}_2)$\\
\L(.55*.07) $\Phi(t_{i-1})$\\%
\L(.55*.68) $\Psi(\bar{t}_{i-1})$\\
\L(.70*.68) $\Psi(t_i)$\\%
\L(.70*.04) $\Phi(\bar{t}_i)$\\
\L(.94*.68) $\Psi(T)=\Phi(T)$\\%
%\L(.40*.97) $\Psi$\\
\endSetLabels
\begin{center}
\AffixLabels{\centerline{\epsfig{file =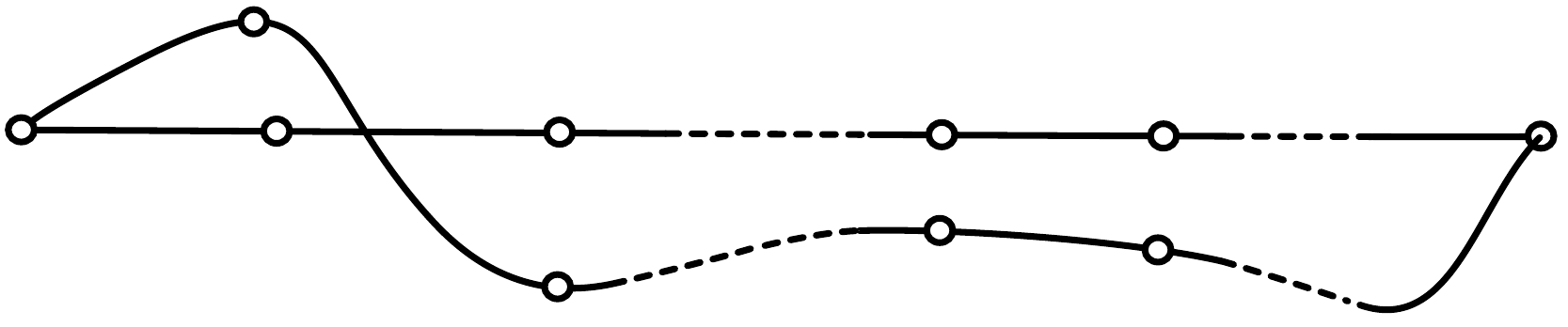,width=12cm,angle=0}}}
\vspace{-24pt}
\end{center}
}
\caption{Proof of Lemma \ref{lem-almost right}: Points on the geodesic $\Psi$ and quasi-geodesic $\Phi$.} \label{fig:Lem22a}
\end{figure}
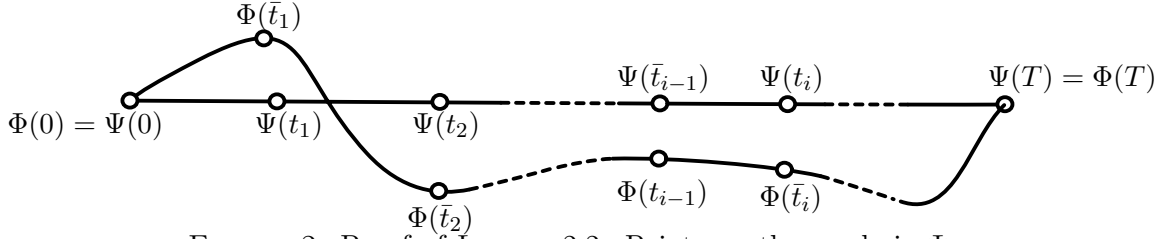

Noting finally that $d_X(\Phi(t_{i-1}),\Phi(t_i))\ge \frac 1LR-A$ and that the latter quantity is positive if $R$ is chosen large enough, we get that
\begin{align*}
d_X(\Phi(0),\Phi(T))&=d_X(\Psi(0),\Psi(T))=\sum_{i=1}^sd_X(\Psi(\bar t_{i-1}),\Psi(\bar t_i))\\
&\ge \sum_{i=1}^sd_X(\Phi(t_{i-1}),\Phi(t_i))-2\cdot D\cdot(s-1)\\
&\ge \sum_{i=1}^sd_X(\Phi(t_{i-1}),\Phi(t_i))-2\cdot D\cdot\sum_{i=1}^s\frac{d_X(\Phi(t_{i-1}),\Phi(t_i))}{\frac 1LR-A}\\
&\ge \left(1-\frac{2\cdot D}{\frac 1LR-A}\right)\sum_{i=1}^sd_X(\Phi(t_{i-1}),\Phi(t_i))
\end{align*}

\begin{figure}[h]
%\ShowGrid
{
\leavevmode \SetLabels
\L(.41*.14) $\Phi(t_{i-1})$\\%
\L(.41*1.06) $\Psi(\bar{t}_{i-1})$\\
\L(.58*1.06) $\Psi(t_i)$\\%
\L(.58*.02) $\Phi(\bar{t}_i)$\\
\L(.45*.6) $\leq D$\\%
\L(.61*.58) $\leq D$\\%
%\L(.40*.97) $\Psi$\\
\endSetLabels
\begin{center}
\AffixLabels{\centerline{\epsfig{file =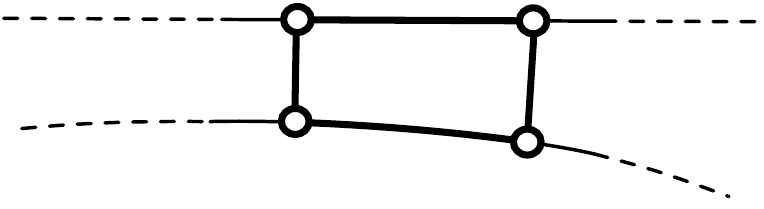,width=6.5cm,angle=0}}}
\vspace{-24pt}
\end{center}
}
\caption{Proof of Lemma \ref{lem-almost right}: Obtaining the inequality $d_X(\Psi(\bar{t}_{i-1},\Psi(\bar{t}_{i})) \geq d_X(\Phi(t_{i-1}),\Phi(t_i))-2\cdot D$.} \label{fig:Lem22b}
\end{figure}
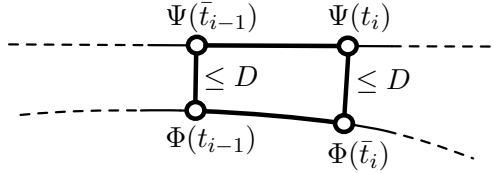

As $D$ only depends on $L$ and $A$, for large enough $R$ we have
$$
\frac{2\cdot D}{\frac 1LR-A} \leq \epsilon
$$
and the claim follows.
\end{proof}

As a consequence we obtain that the {\em length} and {\em stable length}
$$\ell_X(\gamma)=\inf_{x\in X}d_X(x,\gamma(x))\text{ and }\Vert \gamma\Vert_X=\lim_{n\to\infty}\frac 1n\inf_{x\in X}d_X(x,\gamma^n(x))$$
of an isometry $\gamma:X\to X$ basically agree once the former is large enough:

\begin{lem}\label{lem-compare lengths}
Let $X$ be a hyperbolic space. For every $\epsilon>0$ there is some $L(\epsilon)$ such that 
$$\ell_X(\gamma)\ge\Vert\gamma\Vert_X\ge(1-\epsilon)\cdot\ell_X(\gamma)$$
for every isometry $\gamma:X\to X$ with $\ell_X(\gamma)\ge L(\epsilon)$.
\end{lem}

\begin{proof}
Note that the first inequality holds trivially in full generality because $d_X(x,\gamma^n(x))\le n\cdot d_X(x,\gamma(x))$ for any isometry $\gamma$ of a metric space and for every point $x$ in the said space. We concentrate now on the second inequality.

To begin with, let $D$ be such that any $D$-local geodesic in $X$ is an $(L,A)$-quasi-geodesic. Now, given $\epsilon>0$, let $R=R(\delta, L, A, \epsilon)$ so that the conclusion of Lemma \ref{lem-almost right} holds for that $\epsilon$ and all $(L,A)$-quasi-isometries. Finally, set $L(\epsilon)=\max\{D,R\}$ and suppose that $\gamma:X\to X$ is an isometry with
$$\ell_X(\gamma)\ge L(\epsilon)$$
and that $x_0\in X$ is such that $\ell_X(\gamma)=d_X(x_0,\gamma(x_0))$. Fix $\phi:[0,\ell_X(\gamma)]\to X$ to be a geodesic segment with $\phi(0)=x_0$ and $\phi(\ell_X(\gamma))=\gamma(x_0)$. 

Consider now the piecewise geodesic path $\Phi:\BR\to X$ whose restriction to the interval $[k\cdot\ell_X(\gamma),(k+1)\cdot\ell_X(\gamma)]$ is given by $t\mapsto \gamma^k(\phi(t-k\cdot\ell_X(\gamma)))$. This path is not only continuous and piecewise geodesic, but also a $\ell_X(\gamma)$-local geodesic: otherwise we could find a point $y\in\Phi(\BR)$ with $d_X(y,\gamma(y))<d_X(x_0,\gamma(x_0))$, contradicting the choice of $x_0$. 

Anyways, since $\Phi$ is a $D$-local geodesic, it follows from the choice of $D$ that it is also an $(L,A)$-quasi-geodesic. Now, since by construction we also have 
$$d_X(\Phi(k\cdot\ell_X(\gamma)),\Phi((k+1)\cdot\ell_X(\gamma)))=\ell_X(\gamma)\ge R$$
we obtain from Lemma~\ref{lem-almost right} that
\begin{align*}
d_X(x_0,\gamma^{n}(x_0))&=d_X(\Phi(0),\Phi(n\cdot\ell_X(\gamma)))\\
&\geq   (1-\epsilon) \sum_{k=0}^{n-1} d_X(\Phi(k\cdot\ell_X(\gamma)),\Phi((k+1)\cdot\ell_X(\gamma)))\\
&= (1-\epsilon)\cdot n\cdot d_X(\Phi(0),\Phi(\ell_X(\gamma)))\\
&= (1-\epsilon)\cdot n\cdot \ell_X(\gamma).
\end{align*} 
From this, it follows that
$$\Vert\gamma\Vert_X = \lim_{n\to\infty} \frac{1}{n} d_X(x_0,\gamma^{n}(x_0))\ge (1-\epsilon)\cdot\ell_X(\gamma)$$
as desired. We have proved Lemma \ref{lem-almost right}. 
\end{proof}

\section{}\label{sec-currents}
We recall now some facts about the geodesic flow associated to a hyperbolic group, and about the associated space of currents. We refer to Gromov's foundational paper \cite{Gromov} and to \cite{CP} for the construction of the geodesic flow and to \cite{Javi-Chris,Bonahon86,Bonahon88,Bonahon-currents-group} for facts on currents. We point out that currents have been mostly studied in the setting of surface groups (or free groups) but that the pretty general statements we will need below are either discussed in full generality in \cite{Bonahon-currents-group}, or follow using the same argument as in the surface group case.

\subsection{The geodesic flow}
We recall now the upshot of the construction due to Gromov \cite[Section 8]{Gromov} of a geodesic flow for a hyperbolic group. 

Gromov proves that for each hyperbolic group $\Gamma$ there is a proper geodesic metric space $\CU=\CU(\Gamma)$ equipped with an isometric action $\Gamma\actson\CU$, an isometric involution $i:\CU\to\CU$, and a flow $(\phi_t)_{t\in\BR}$ satisfying
$$\gamma(\phi_t(v))=\phi_t(\gamma(v)),\ i(\phi_t(v))=\phi_{-t}(i(v)),\text{ and }\gamma(i(v))=i(\gamma(v))$$
for all $\gamma\in\Gamma$, $t\in\BR$ and $v\in\CU$.
Moreover, the action $\Gamma\actson\CU$ is properly discontinuous and cocompact. In particular, $\CU$ is quasi-isometric to $\Gamma$ and thus hyperbolic. It follows that we can identify $\D\CU$ and $\D\Gamma$.

Furthermore, the flow $(\phi_t)$ has the property that for all $v\in\CU$ the map $\BR\to\CU$, $t\mapsto\phi_t(v)$ is a quasi-isometric embedding with uniform constants in $v$. In particular, to every $v$, one can associate two points 
$$v^+=\lim_{t\to\infty}\phi_t(v)\text{ and }v^-=\lim_{t\to-\infty}\phi_t(v)$$
in $\D\CU=\D\Gamma$.
A feature---possibly the key feature---of the flow $\phi_t$ is that the fibers of the map
\begin{equation}\label{eq-endpoints}
\CU\to\D^2\Gamma,\ \ v\mapsto(v^+,v^-)
\end{equation}
are precisely the orbits of the flow $\phi_t$. In fact, the map \eqref{eq-endpoints} is the first coordinate of a homeomorphism $\CU\simeq\D^2\Gamma\times\BR$.

Finally, the space $\CU$ is canonical in the following sense. If $(\CU',\phi_t',i')$ is any further space with these properties, then there is a map $\CU\to\CU'$ which commutes with the actions of $\Gamma$, which conjugates $i$ and $i'$, and which maps flowlines of $\phi_t$ to flowlines of $\phi_t'$---said differently, the space $\CU$ is unique up to orbit equivalence.

\begin{bem}
We have chosen to denote the space by $\CU$ to remind the reader of the {\em unit} tangent bundle of a Riemannian manifold $M$, for it is in that space on which the geodesic flow lives. The involution $i$ is in that case given by $i(v)=-v$.
\end{bem}

Note that if the group $\Gamma$ is torsion free, then the discreteness of the action $\Gamma\actson\CU$ implies that it is actually free.

\subsection{Transversals}
Continuing with the same notation, suppose now that $\Gamma$ is not only hyperbolic but also torsion free. By a {\em transversal} for the geodesic flow we mean a compact subset $\tau\subset\CU$ with the following properties:
\begin{itemize}
\item The map $\tau\times[-\epsilon,\epsilon]\to\CU$, $(v,t)\mapsto\phi_t(v)$ is an embedding for some $\epsilon>0$.
\item The restriction of the covering map $\pi:\CU\to\CU/\Gamma$ to $\tau$ is an embedding.
\item For every $v\in\CU$ there are $t>0$ and $g\in\Gamma$ with $\phi_t(v)\in g(\tau)$.
\end{itemize}
We make two comments which will come in handy later on:
\begin{enumerate}
\item A transversal $\tau$ has, as a subset of $\CU$, empty interior. We will say however that a point $x\in\tau$ is an interior point if, for all $\epsilon>0$ sufficiently small, it is an interior point of the image of the map $\tau\times[-\epsilon,\epsilon]\to\CU$, $(v,t)\mapsto\phi_t(v)$. A point which is not interior is a boundary point and we denote by $\D\tau$ the set of all boundary points.
\item A measure $\mu$ on $\CU$ invariant under $\Gamma$ and the flow $\phi_t$ induces a measure $\mu^\tau$ with
$$\mu^\tau(U)=\epsilon^{-1}\cdot\mu(\{\phi_t(v)\vert t\in[0,\epsilon],v\in U\})$$
for all sufficiently small $\epsilon>0$ and all $U\subset\tau$ open.
\end{enumerate}

After these comments, we discuss briefly how compactness of $\CU/\Gamma$ and the fact that $\pi$ is a covering map imply that transversals exist. In fact, as in \cite{Bonahon-currents-group}, compactness of $\CU/\Gamma$ implies that one can cover $\CU/\Gamma$ with finitely many flow-boxes $B_i\simeq A_i\times[0,1]$. Here $A_i\subset\CU/\Gamma$ and the embedding $A_i\times[0,1]\hookrightarrow B_i\subset\CU/\Gamma$ is given by $(a,t)\mapsto\phi_t(a)$. Moreover, one can choose the sets $A_i$ to be disjoint of each other and of small diameter. This last property implies that they lift homeomorphically to sets $\tilde A_i\subset\CU$. The set $\tau=\cup_i\tilde A_i$ is a transversal.

Note that the same argument shows that one can construct two transversals $\tau\subset\hat\tau\subset\CU$ with the property that $\tau\cap\D\hat\tau=\emptyset$. Now, for all $t\in[0,\epsilon]$ we have that the set
$$\tau_t=\{v\in\hat\tau\,\vert\, d_{\CU}(v,\tau)\le t\}$$
is again a transversal. Moreover, since $\epsilon$ is small, we have that 
$$\D\tau_{t_1}\cap\D\tau_{t_2}=\emptyset$$
for all $t_1\neq t_2$ positive and smaller than $\epsilon$. Since we have an uncountable collection of disjoint sets, most of them have to have measure $0$ with respect to any locally finite measure on $\tau$. This means that for any given measure one can get transversals whose boundaries have vanishing measure. We state this formally as follows:

\begin{lem}\label{lem-good transversal}
If $\mu$ is a locally finite measure on $\CU$ invariant under $\Gamma$ and the flow $\phi_t$, then there is a transversal $\tau\subset\CU$ with $\mu^\tau(\D\tau)=0$.\qed
\end{lem}

\subsection{First return map}
Again using the same notation as above, suppose that $\tau\subset\CU$ is a transversal for the flow $\phi_t$. Recalling that for every $v\in\tau$ there is some $t>0$ and $g\in\Gamma$ with $\phi_t(v)\in g(\tau)$, we define a number of maps. First we consider the $n$-th time that the orbit of $v\in\tau$ meets $\Gamma\tau$. More concretely, given $n\in\BN$ and $v\in\tau$ we set $\rho(0,v)=0$ and define inductively
$$\rho(n,v)=\min\{t>\rho(n-1,v)\text{ such that there is a }g\in\Gamma\text{ with }\phi_t(v)\in g\tau\}.$$
In this way we have a map
$$\rho:\BN\times\tau\to\BR_+,\ (n,v)\mapsto\rho(n,v)$$
and we refer to $\rho(n,v)$ as the {\em $n$-th return time} of $v$. We will also be interested in the map
$$T:\BN\times\tau\to\CU,\ \ (n,v)\mapsto T_n(v)\stackrel{\text{def}}=\phi_{\rho(n,v)}(v).$$
By construction, this map takes values in the subset $\Gamma\tau\subset\CU$. We have thus for all $(n,v)\in\BN\times\tau$ an element $g_n(v)\in\Gamma$ with
$$T_n(v)=\phi_{\rho(n,v)}(v)\in g_n(v)\tau.$$
Note that the second condition in the definition of transversal implies that $g_n(v)$ is indeed unique.

\begin{figure}[h]
%\ShowGrid
{
\leavevmode \SetLabels
\L(.14*.0) $\tau$\\%
\L(.29*-.02) $g_1(v)(\tau)$\\%
\L(.50*-.03) $g_2(v)(\tau)$\\%
\L(.64*-.05) $g_{i-1}(v)(\tau)$\\%
\L(.84*-.02) $g_{i}(v)(\tau)$\\%
\L(.1*.64) $v$\\%
\L(.36*.42) $T_1(v)$\\%
\L(.578*.15) $T_2(v)$\\
\L(.72*.39) $T_{i-1}(v)$\\
\L(.868*.77) $T_i(v)$\\
%\L(.40*.97) $\Psi$\\
\endSetLabels
\begin{center}
\AffixLabels{\centerline{\epsfig{file =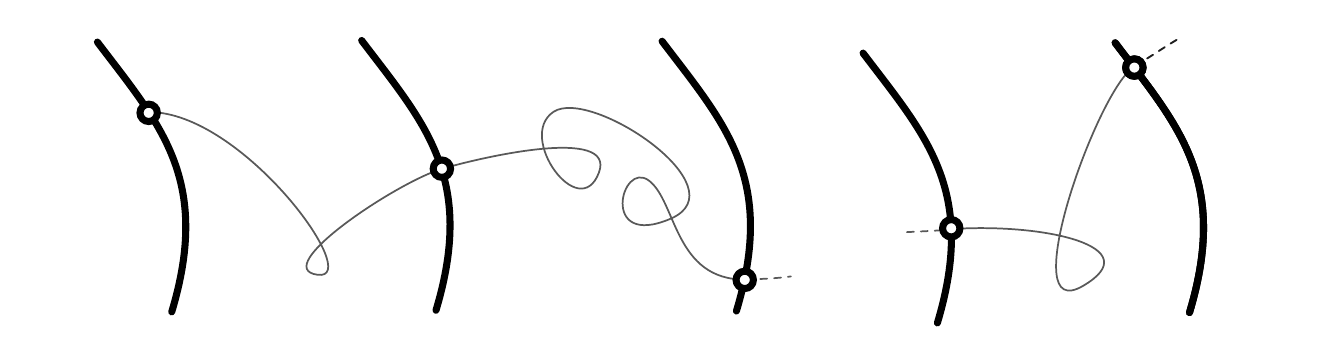,width=12.0cm,angle=0}}}
%\vspace{-12pt}
\end{center}
}
{
\leavevmode \SetLabels
\L(.44*.7) $v$\\%
\L(.295*.52) $g_1(v)^{-1}(T_1(v))$\\
\L(.53*.17) $g_2(v)^{-1}(T_2(v))$\\
\L(.535*.35) $g_{i-1}(v)^{-1}(T_{i-1}(v))$\\
\L(.23*.85) $g_i(v)^{-1}(T_i(v))$\\
\endSetLabels
\begin{center}
\AffixLabels{\centerline{\epsfig{file =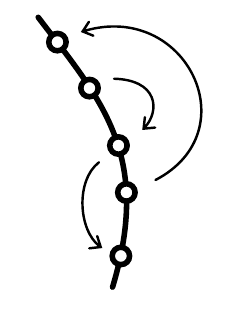,width=3.0cm,angle=0}}}
\vspace{-24pt}
\end{center}
}
\caption{The maps $g_n, T_n$ (above) and $P$ (below)%Obtaining the inequality
%$$d_X(\Psi(\bar{t}_{i-1},\Psi(\bar{t}_{i})) \geq d_X(\Phi(t_{i-1}),\Phi(t_i))-2\cdot D$$
} \label{fig:Return1}
\end{figure}
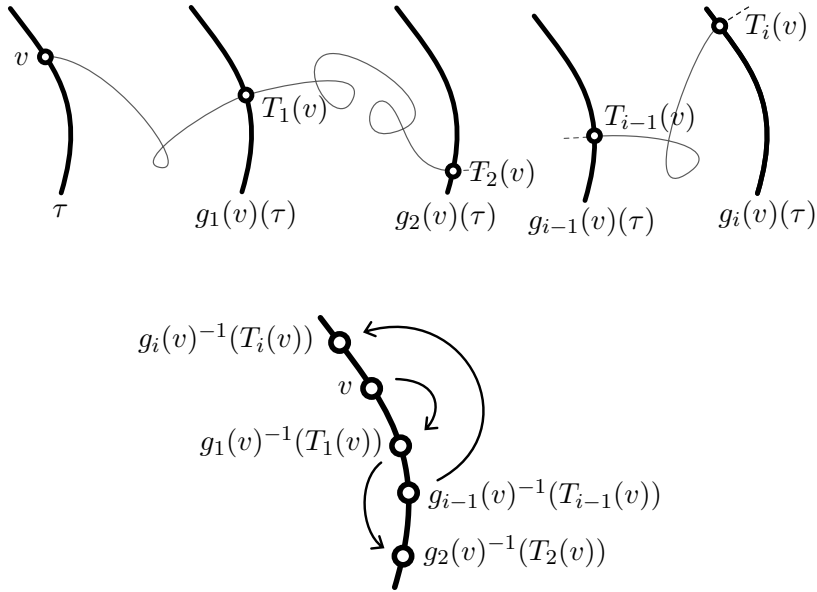

Anyways, we have a map
$$g:\BN\times\tau\to\Gamma,\ (n,v)\mapsto g_n(v)$$
which better remains unnamed. Finally, we have the {\em first return map}
$$P:\tau\to\tau,\ \ P(v)=g_1(v)^{-1}T_1(v).$$
More generally, we have the following cocycle equation:
\begin{equation}\label{cocycle}
T_{n+k}(v)=g_k(v)\cdot T_n(P^k(v)).
\end{equation}

Note also that it follows directly from the definitions that, for fixed $n$, the maps $\rho_n(\cdot),T_n(\cdot),g_n(\cdot)$ and $P^n(\cdot)$ are continuous at $v\in\tau$ unless one of the points $v,P(v),\dots,P^{n-1}(v)$ belongs to $\D\tau$. Combining this observation with Lemma \ref{lem-good transversal} we get:

\begin{lem}\label{lem-very good transversal}
Let $\mu$ be a locally finite measure on $\CU$ invariant under $\Gamma$ and the flow $\phi_t$. Then there is a transversal $\tau\subset\CU$ with $\mu^\tau(\D\tau)=0$ and such that for all $n$ there is a closed set $\sigma_n\subset\tau$ with $T_n(\cdot)$ continuous on $\tau\setminus\sigma_n$ and with $\mu^\tau(\sigma_n)=0$.\qed
\end{lem}

\begin{bem}
Given how often we used the word ``return" it is likely that the reader has figured out that we have been secretly thinking of the flow induced by $\phi_t$ on $\CU/\Gamma$ and that we have been identifying $\tau$ with its projection to $\CU/\Gamma$. This is correct. However, there are two reasons not to present things from that point of view. First, the map $g:\BN\times\tau\to\Gamma$ is less natural if we work in the quotient. Second, we would be forced to either introduce additional notation for the maps in the quotient, or to identify maps in the quotient and in the universal cover, risking to confuse the reader and ourselves.
\end{bem}

\subsection{Oriented currents}
We continue with the same notation as all along. An {\em oriented current} is a $\Gamma$-invariant Radon measure on the double boundary $\D^2\Gamma$ of $\Gamma$. Recall that a Borel measure is Radon if it is locally finite and inner regular. We denote the space of all oriented currents, endowed with the weak-*-topology, by $\CC^+(\Gamma)$. The space $\CC^+(\Gamma)$ of oriented currents is a cone in a linear space.

Recall now that every automorphism $\Gamma\to\Gamma$ induces a homeomorphism $\D^2\Gamma\to\D^2\Gamma$. It thus induces a homeomorphism $\CC^+(\Gamma)\to\CC^+(\Gamma)$ between the spaces of oriented currents. It is easy to see that this map is the identity if the automorphism we began with was inner. In particular, the group $\Out(\Gamma)=\Aut(\Gamma)/\Inn(\Gamma)$ of inner automorphisms acts on $\CC^+(\Gamma)$.

We now give examples of oriented currents. We can associate to the conjugacy class of every non-trivial primitive group element $\gamma\in\Gamma$ an element in $\CC^+(\Gamma)$ as follows. The element $\gamma$ has, when acting on $\D\Gamma$, a unique attractive and a unique repelling fixed point $\gamma_{+}$ and $\gamma_-$. They are different, meaning that $(\gamma_{-},\gamma_{+})\in\partial^2\Gamma$. The $\Gamma$-orbit $\Gamma(\gamma_-,\gamma_+)$ of this point is a discrete subset of the double boundary $\partial^2\Gamma$ and depends only on the conjugacy class of $\gamma$. Abusing terminology, we denote by $\gamma$ also the oriented current given by the atomic measure supported on $\partial^2\Gamma$, where each point in $\Gamma(\gamma_-,\gamma_+)$ is an atom of weight one. If we denote by $\CS^+(\Gamma)$ the set of all conjugacy classes in $\Gamma$, then it is a theorem by Bonahon \cite{Bonahon-currents-group} that  
$$\mathbb{R}_+\CS^+(\Gamma)=\{t\cdot\gamma\,|\,t\in\mathbb{R}_+, \gamma\in\CS^+(\Gamma)\}$$ 
is dense in $\CC^+(\Gamma)$.

\begin{bem}
Recall now that $\Gamma$ is supposed to be torsion free. This implies that the stabilizers of points in $\D\Gamma$ are not only virtually cyclic, but actually cyclic. In particular we could, as was done in \cite{Bonahon-currents-group}, associate to a non-primitive element a current as follows: write $\gamma$ as $\eta^n$ with $\eta$ primitive and $n$ positive and associate to $\gamma$ the current $n\cdot\eta$.
\end{bem}

Finally note that via the homeomorphism
$$\CU\simeq \D^2\Gamma\times\BR$$
we can identify oriented currents with measures on $\CU$ invariant under $\Gamma$ and under the flow $\phi_t$. More concretely, if $\mu$ is an oriented current then there is a unique measure $\hat\mu$ on $\CU$ that satisfies the following condition: if $A\subset\CU$ is such that 
$$\{\phi_t(v)\vert t\in\BR\}\cap A=\{v\}$$
for all $v\in A$ then
$$\hat\mu(\cup_{s\in[0,t]}\phi_s(A))=t\cdot\mu(\{(v^+,v^-)\in\D^2\Gamma\,\vert\, v\in A\}).$$
Note also that $\hat\mu$, being invariant under $\phi_t$, induces a measure $\mu^\tau$ on every transversal $\tau$. This measure can be directly described as follows:
$$\mu^\tau(A)=\mu(\{(v^+,v^-)\in\D^2\Gamma\,\vert\, v\in A\})$$
for every $A\subset\tau$ measurable with $\{\phi_t(v)\vert t\in\BR\}\cap A=\{v\}$ for all $v\in A$.

\subsection{Currents and the flip}
The double boundary $\D^2\Gamma$ is endowed with the so-called {\em flip} $(\theta,\eta)\mapsto(\eta,\theta)$. A {\em current} is a $\Gamma$-invariant Radon measure on $\D^2\Gamma/{\text{flip}}$. Equivalently, a current is a flip-invariant oriented current. From the point of view of measures on $\CU$, geodesic currents correspond to measures which are not only invariant under the flow $\phi_t$ but also under the involution $i:\CU\to\CU$. 

In general, the oriented current associated to the conjugacy class of $\gamma\in\Gamma$ is not flip-invariant: its image is namely the current associated to $\gamma^{-1}$. However, one can associate to every oriented current $\mu$ the flip-invariant current $\frac 12(\mu+i_*\mu)$. In this way we associate to the conjugacy class of a primitive element $\gamma\in\Gamma$ a current which we denote once again by $\gamma$. 

\begin{bem}
The reason to work with both currents and oriented currents is the following. From a technical point of view, more concretely from the point of view of bookkeeping, working with oriented currents is slightly simpler. Possibly, oriented currents are in general also more natural. However, in some situations one is interested in objects which are intrinsically unoriented---such as laminations on surfaces. This is why most of the literature is about currents \cite{Javi-Chris,Bonahon86,Bonahon88}. In any case, as observed in \cite{Bonahon-currents-group}, it is easy to pass between these points of view.
\end{bem}

\section{}\label{sec-meat}
Continuing with the notation of the previous section, let $\Gamma$ be a torsion free hyperbolic group and $(\CU,\phi_t,i)$ the corresponding geodesic flow. Our next goal is to define the {\em stable length of an oriented current with respect to an action $\Gamma\actson X$} and prove that it is continuous. More precisely we prove Theorem \ref{sat-length} from the introduction, which we recall here for convenience:

\begin{named}{Theorem \ref{sat-length}}
Let $\Gamma\actson X$ be a discrete and cocompact isometric action of a torsion free Gromov hyperbolic group  on a geodesic metric space and let $\CC^+(\Gamma)$ be the space of oriented currents of $\Gamma$. There is a unique continuous function
$$\Vert\cdot\Vert_X:\CC^+(\Gamma)\to\BR_+$$
with $\Vert\gamma\Vert_X=\lim_{n\to\infty}\frac 1n\inf_{x\in X}d_X(x,\gamma^n(x))$ for every non-trivial $\gamma\in\Gamma$. Moreover, the function $\Vert\cdot\Vert_X$ is flip-invariant.
\end{named}

First, recall that by Lemma \ref{lem-cone off} the space $X$ is a convex subset of a contractible hyperbolic space $\hat X$ on which our group $\Gamma$ still acts. Since $X$ is convex in $\hat X$ we have that $d_X(x,x')=d_{\hat X}(x,x')$ for all $x,x'\in X\subset\hat X$. In particular, the stable lengths of $\gamma\in\Gamma$ with respect to $X$ and $\hat X$ agree:
$$\Vert\gamma\Vert_X=\Vert\gamma\Vert_{\hat X}.$$
We can thus assume that $X$ was contractible to begin with.

After this preliminary (and not very important) comment, we start with the discussion of the proof of Theorem \ref{sat-length}. Possibly, the main difficulty when proving this theorem is to find a candidate for $\Vert\mu\Vert_X$. To construct such a candidate we choose a transversal $\tau\subset\CU$ for the geodesic flow and recall the definition of the map 
$$T:\BN\times\tau\to\CU$$
given above. Note also that since $\Gamma$ acts cocompactly and by isometries on both $\CU$ and $X$, it follows from the Milnor-Schwarz theorem that there is a $\Gamma$-equivariant quasi-isometry 
$$\Phi:\CU\to X.$$
The assumption that $X$ is contractible implies that $\Phi$ can be chosen to be continuous---from a technical point this is a useful property. Anyways, given an oriented current $\mu\in\CC^+(\Gamma)$ and $k\ge 1$ we consider the quantity
$$L_k(\mu,\tau,\Phi)=\int_{\tau}d_X(\Phi(v),\Phi(T_k(v)))d\mu^\tau(v).$$
Here, as was the case earlier, $\mu^\tau$ is the measure induced by the current $\mu$ on the transversal $\tau$.

\begin{lem}\label{lem-subadditive}
The sequence $k\to L_k(\mu,\tau,\Phi)$ is sub-additive.
\end{lem}

\begin{proof}
Recall that in addition to the map $T$ we also defined the first return map $P:\tau\to\tau$ of the geodesic flow and the unnamed map $g:\BN\times\tau\to\Gamma$, and that these maps are related by  equation \eqref{cocycle}, which we recall here:
$$T_{n+k}(v)=g_k(v)\cdot T_n(P^k(v)).$$
As such
\begin{align*}
d_X(\Phi(v),\Phi(T_{n+k}(v)))=
&d_X(\Phi(v),\Phi(g_k(v)\cdot T_n(P^k(v))))\\
\le& d_X(\Phi(v),\Phi(g_k(v)\cdot P^k(v)))\\
&+d_X(\Phi(g_k(v)\cdot P^k(v)),\Phi(g_k(v)\cdot T_n(P^k(v))))\\
=& d_X(\Phi(v),\Phi(T_k(v)))+d_X(\Phi(P^k(v)),\Phi(T_n(P^k(v))))
\end{align*}
where the last equation holds because of the equivariance of $\Phi$. From this we get that
\begin{align*}
L_{n+k}(\mu,\tau,\Phi)=& \int_{\tau}d_X(\Phi(v),\Phi(T_{n+k}(v)))d\mu^\tau(v)\\
\le & \int_{\tau}\left(d_X(\Phi(v),\Phi(T_k(v)))+d(\Phi(P^k(v)),\Phi(T_n(P^k(v))))\right)d\mu^\tau(v)\\
=& \int_{\tau}d_X(\Phi(v),\Phi(T_k(v)))d\mu^\tau(v)\\
&+\int_{\tau}d_X(\Phi(P^k(v)),\Phi(T_n(P^k(v))))d\mu^\tau(v)\\
=& \int_{\tau}d_X(\Phi(v),\Phi(T_k(v)))d\mu^\tau(v)+\int_{\tau}d_X(\Phi(v),\Phi(T_n(v)))d\mu^\tau(v)\\
=&L_k(\mu,\tau,\Phi)+L_n(\mu,\tau,\Phi)
\end{align*}
where the third equality holds because $P$ preserves the measure $\mu^\tau$.
\end{proof}

It follows from Lemma \ref{lem-subadditive} and the Fekete lemma that $\lim_{k\to\infty}\frac{L_k(\mu,\tau,\Phi)}k$ exists. In fact, it is independent of the chosen quasi-isometry. To see why this is the case recall that any other $\Gamma$-equivariant isometry $\Psi:\CU\to X$ is at bounded distance from $\Phi$. Let $D$ be such a bound. Then we have for all $k$ and all $v\in\tau$ that
$$\left\vert d_X(\Phi(v),\Phi(T(k,v)))-d_X(\Psi(v),\Psi(T(k,v)))\right\vert\le 2\cdot D.$$
It follows that
$$\left\vert L_k(\mu,\tau,\Phi)-L_k(\mu,\tau,\Psi)\right\vert\le 2\cdot D\cdot\mu(\tau).$$
Since this error does not depend on $k$ we get that it vanishes once we divide by $k$ and let it grow. Altogether we have:

\begin{kor}\label{kor-stable length exists}
Let $\mu\in\CC^+(\Gamma)$ be an oriented current, $\tau\subset\CU$ a transversal and $\Phi:\CU\to X$ a $\Gamma$-equivarient quasi-isometry. Then the limit
$$\CL(\mu,\tau)=\lim_{k\to\infty}\frac{L_k(\mu,\tau,\Phi)}k$$
exists and does not depend on $\Phi$. \qed
\end{kor}

The limit in Corollary \ref{kor-stable length exists} is our candidate for the stable length $\Vert\mu\Vert_X$ of an oriented current $\mu\in\CC^+(\Gamma)$. Indeed, we note next that this is the case for currents induced by individual conjugacy classes:

\begin{lem}\label{lem-curves}
Let $\tau\subset\CU$ be a transversal. We have $\CL(\gamma,\tau)=\Vert\gamma\Vert_X$ for every $\gamma\in\CS^+(\Gamma)$.
\end{lem}
\begin{proof}
Choose a representative of the conjugacy class $\gamma$ and denote it again by $\gamma\in\Gamma$. Let $(\gamma^-,\gamma^+)\in\D^2\Gamma=\D^2\CU$ be the pair of its attracting and repelling fixed points, and recall that there is a unique orbit $\phi_t(v_0)$ of the flow $\phi_t$ with endpoints $(\gamma^-,\gamma^+)$:
$$\gamma^+=\lim_{t\to\infty}\phi_t(v_0)\text{ and }\gamma^-=\lim_{t\to-\infty}\phi_t(v_0).$$
Uniqueness implies that this flow line is periodic under $\gamma$. Note also that up to replacing $\gamma$ by a conjugate, we can assume that the orbit $\phi_t(v_0)$ meets $\tau$. Then, up to replacing $v_0$ by another point of the same orbit, we can actually assume that $v_0\in\tau$.

For $n\ge 0$ we set $v_n=T_n(v_0)$ and note that there is some $k$ with $\gamma\cdot v_{0}=v_k$. Geometrically, $k$ is the number of times that the projection of the flow line of $v_0$ to $\CU/\Gamma$ meets the projection of $\tau$. Since the flow line is $\gamma$-invariant we also have that $\gamma\cdot v_{n}=v_{n+k}$ for all $n$. Note that this implies that $\tau$ meets the $\Gamma$-translates of the flow-line $\phi_t(v_0)$ exactly at the points $v_0,g_1(v_0)^{-1}v_1,\dots,g_{k-1}(v_0)^{-1}v_{k-1}$. Hence the measure 
$$\mu^\tau=\sum_{i=0}^{k-1}\delta_{g_i(v_0)^{-1}v_i}$$ 
is the sum of Dirac measures of weight 1 centred at those points. Altogether we have
\begin{align*}
L_m(\gamma,\mu,\Phi)
&=\sum_{i=0}^{k-1}d_X(\Phi(g_i(v_0)^{-1}v_i),\Phi(T_m(g_i(v_0)^{-1}v_i)))\\
&=\sum_{i=0}^{k-1}d_X(\Phi(v_i),g_i(v_0)\Phi(T_m(g_i(v_0)^{-1}v_i)))\\
&=\sum_{i=0}^{k-1}d_X(\Phi(v_i),\Phi(v_{m+i})).
\end{align*}
Now, if $m$ is divisible by $k$ we have
\begin{align*}
L_{nk}(\gamma,\mu,\Phi)&=\sum_{i=0}^{k-1}d_X(\Phi(v_i),\Phi(v_{kn+i}))=\sum_{i=0}^{k-1}d_X(\Phi(v_i),\Phi(\gamma^nv_i))\\
&=\sum_{i=0}^{k-1}d_X(\Phi(v_i),\gamma^n\Phi(v_i)).
\end{align*}
Since the difference between $d_X(\Phi(v_i),\gamma^n\Phi(v_i))$ and $d_X(\Phi(v_0),\gamma^n\Phi(v_0))$ is bounded independently of $n$, and since the last sum above has $k$ summands, we get that there is a constant $C$ with
$$\left\vert L_{nk}(\gamma,\mu,\Phi)-k\cdot d_X(\Phi(v_0),\gamma^n\Phi(v_0))\right\vert\le C$$
for all $n$. It follows that
$$\CL(\mu,\gamma)=\lim_{n\to\infty}\frac 1{nk}L_{nk}(\gamma,\mu,\Phi)=\lim_{n\to\infty}\frac 1{nk}k\cdot d_X(\Phi(v_0),\gamma^n\Phi(v_0))=\Vert\gamma\Vert_X,$$
as we needed to prove.
\end{proof}

Our next goal is to prove that the limits in Corollary \ref{kor-stable length exists} are uniform:

\begin{lem}\label{lem-uniform}
Let $\Phi:\CU\to X$ be a $\Gamma$-equivariant quasi-isometry and $\tau\subset\CU$ a transversal. For every $\epsilon>0$ there is $n_0$ such that for every oriented current $\mu\in\CC^+(\Gamma)$ one has
$$\left\vert\CL(\mu,\tau)-\frac{L_n(\mu,\tau,\Phi)}{n}\right\vert\le\epsilon\cdot \mu^\tau(\tau)$$
for all $n\ge n_0$.
\end{lem}

\begin{proof}
Note that, by the existence of the limit in Corollary \ref{kor-stable length exists}, it suffices to prove that there is $n_0$ with 
$$\frac 1n\left\vert L_n(\mu,\tau,\Phi)-\lim_{k\to\infty}\frac{L_{nk}(\mu,\tau,\Phi)}k\right\vert\le\epsilon\cdot \mu^\tau(\tau)$$
for all $n\ge n_0$. Since the number in the absolute value is in fact positive due to Lemma \ref{lem-subadditive}, it suffices to bound it from above by $\epsilon\cdot \mu^\tau(\tau)$.

Let $\epsilon'$ be small and $K$ large; how small and how large will be determined later on. Armed with $K$, choose $n_0$ with $d_{\CU}(v,T_{n}(v))\ge K$ for all $v\in\tau$ and for all $n\geq n_0$ (such an $n_0$ exists by compactness of $\tau$). We have that also
$$d_{\CU}\left(T_{n(i-1)}(v),T_{ni}(v)\right)\ge K\text{ for all }v\in\tau,\ n\ge n_0\text{ and }i=1,\dots,k.$$
Now, we can assume, using Lemma \ref{lem-almost right}, that $K$ was chosen large enough so that 
$$(1-\epsilon')\sum_{i=1}^{k}d_X(\Phi\left(T_{n(i-1)}(v)),\Phi(T_{ni}(v))\right)\leq d_X\left(\Phi(v), \Phi(T_{nk}(v))\right)$$%<\sum_{i=1}^{k}d_X\left(\Phi(T_{n(i-1)}(v),\Phi(T_{ni}(v))\right)$$
for all $v\in\tau$ and all $n\ge n_0$. Now \eqref{cocycle} implies that
\begin{align*}
&\Phi(T_{n(i-1)}(v))=\Phi(g_{n(i-1)}(v)P^{n(i-1)}(v))=g_{n(i-1)}(v)\Phi(P^{n(i-1)}(v))\\
&\Phi(T_{ni}(v))=\Phi(g_{n(i-1)}(v)T_n(P^{n(i-1)}(v)))=g_{n(i-1)}(v)\Phi(T_n(P^{n(i-1)}(v)))
\end{align*}
and thus that
$$d_X\left(\Phi(T_{n(i-1)}(v)), \Phi(T_{ni}(v))\right)=d_X\left(\Phi(P^{n(i-1)}(v)),\Phi(T_n(P^{n(i-1)}(v)))\right)$$
for all $i=1,\ldots k$ and all $v\in\tau$. Integrating over $\tau$ we get
\begin{align*}
L_{nk}(\mu,\tau,\Phi)&=\int_{\tau}d_X\left(\Phi(v),\Phi(T_{nk}(v))\right)d\mu^{\tau}(v)\\
&\ge(1-\epsilon')\cdot \int_{\tau}\sum_{i=1}^{k}d_X\left(\Phi(P^{n(i-1)}(v)),\Phi(T_n(P^{n(i-1)}(v)))\right)d\mu^{\tau}(v)\\
&=(1-\epsilon')\cdot \sum_{i=1}^{k}\int_{\tau}d_X\left(\Phi(v), \Phi(T_n(v))\right)d\mu^{\tau}(v)\\
& =(1-\epsilon')\cdot  k \int_{\tau}d_X\left(\Phi(v), \Phi(T_n(v))\right)d\mu^{\tau}(v)\\
&=(1-\epsilon')\cdot  k\cdot L_n(\mu, \tau, \Phi).\\ 
\end{align*}
In particular we have for all $n\ge n_0$ and all $k$ that
$$\frac 1n\left(L_n(\mu,\tau,\Phi)-\frac{L_{nk}(\mu,\tau,\Phi)}k\right)<\epsilon'\cdot \frac{1}{n}L_n(\mu,\tau,\Phi)< \epsilon'\cdot L_1(\mu,\tau,\Phi)$$
where the last inequality follows from Lemma \ref{lem-subadditive}. 

If we set $M=\max_{v\in\tau}\left\{d_X\left(\Phi(v),\Phi(T(v))\right)\right\}$ we get that
$$\frac 1n\left(L_n(\mu,\tau,\Phi)-\frac{L_{nk}(\mu,\tau,\Phi)}k\right)<\epsilon'\cdot \frac{1}{n}L_n(\mu,\tau,\Phi)< \epsilon'\cdot M\cdot\mu_\tau(\tau).$$ 
The claim now follows by choosing $\epsilon'=\frac\epsilon M$.
\end{proof}

Armed with Lemma \ref{lem-uniform} we can now prove that for every oriented current there exists a transversal $\tau$ such that $\CL(\cdot,\tau)$ is continuous at that point:

\begin{lem}\label{lem-continuity}
For every $\mu_0\in\CC^+(\Gamma)$ there is a transversal $\tau$ such that $\CL(\cdot,\tau)$ is continuous at $\mu_0$. 
\end{lem}

\begin{proof}
Recall that by Lemma \ref{lem-very good transversal} (and from the equivalence between oriented currents and measures invariant under the flow $\phi_t$), there is a transversal $\tau\subset\CU$ with $\mu_0^\tau(\D\tau)=0$ and such that for all $n$ there is a closed set $\sigma_n\subset\tau$ with $T_n(\cdot)$ continuous on $\tau\setminus\sigma_n$ and with $\mu_0^\tau(\sigma_n)=0$.

Suppose that $\left(\mu_i\right)$ is a sequence of measures converging to $\mu_0$. The condition $\mu_0^\tau(\D\tau)=0$ implies that the measures $\mu_i^\tau$ converge to the measure $\mu_0^\tau$. Now fix $n$. The fact that $T_n(\cdot)$ is continuous outside of a closed set of vanishing $\mu_0^\tau$-measure, together with the continuity of the quasi-isometry $\Phi$, implies that 
$$\tau\to\BR,\ \ v\mapsto d_X\left(\Phi(v), \Phi(T_n(v))\right)$$
is also continuous outside of a closed set of vanishing $\mu_0^\tau$-measure. This shows that $\int d_X\left(\Phi(v), \Phi(T_n(v))\right)d\mu_i^\tau$ converges to $\int d_X\left(\Phi(v), \Phi(T_n(v))\right)d\mu_0^\tau$, which in our terminology just means that
$$\lim_{i\to\infty}L_n(\mu_i,\tau,\Phi)=L_n(\mu_0,\tau,\Phi)$$
for all $n$. 

Let $\epsilon>0$ and chose $n$ larger than the $n_0$ given by Lemma \ref{lem-uniform}. For all sufficiently large $i$ we have
$$\mu_i^\tau(\mu_i)\le 2 \mu_0^\tau(\tau)$$
and 
$$\vert L_n(\mu_i,\tau,\Phi)-L_n(\mu_0,\tau,\Phi)\vert\le\epsilon.$$
From Lemma \ref{lem-uniform} and the triangular inequality we get that 
$$\left\vert \CL(\mu_0,\tau)- \CL(\mu_i,\tau) \right\vert \le(1+3\mu^{\tau_0}(\tau_0))\cdot\epsilon.$$
Since $\epsilon$ was arbitrary, the claim follows.
\end{proof}

We are finally ready to prove Theorem \ref{sat-length}:

\begin{proof}[Proof of Theorem \ref{sat-length}]
We consider first the stable length function defined on $\BR_+\CS^+(\Gamma)$ by $\Vert t\gamma\Vert_X=t\cdot\Vert\gamma\Vert_X$. Since Bonahon \cite{Bonahon-currents-group} proved that $\BR_+\CS^+(\Gamma)$ is dense in $\CC^+(\Gamma)$, we can consider the associated limsup and liminf functions
$$f^{\sup}, f^{\inf}: \CC^+(\Gamma)\to\mathbb{R}^+.$$ 
These are the functions given by
$$f^{\sup}(\mu)=\limsup\Vert t\gamma\Vert_X\text{ and }f^{\inf}(\mu)=\liminf\Vert t \gamma\Vert_X,$$
where the limits are taken over $t\gamma\in\BR_+\CS^+(\Gamma)$ with $t\gamma\to\mu$. We have to prove that $f^{\sup}(\mu)=f^{\inf}(\mu)$ for all $\mu$.

To see that this is the case fix $\mu\in\CC^+(\Gamma)$ and let $\tau$ be a transversal provided by Lemma \ref{lem-continuity}. Also let $(t_i\gamma_i)$ be an arbitrary sequence in $\mathbb{R}_+\CS^+(\Gamma)$ converging to $\mu$ and such that
$$f^{\sup}(\mu)=\lim_{i\to\infty}\Vert t_i\gamma_i\Vert_X$$
Then we have 
$$f^{\sup}(\mu)= \lim_{i\to\infty}t_i\Vert\gamma_i\Vert_X=\lim_{i\to\infty}t_i\CL(\gamma_i,\tau)=\lim_{i\to\infty}\CL(t_i\gamma_i,\tau)=\CL(\mu,\tau)$$
where the second equality holds by Lemma \ref{lem-curves}, the third because $\CL(\cdot,\tau)$ is homogenous and the fourth because $\mu$ is a point of continuity of $\CL(\cdot,\tau)$ by the choice of $\tau$. The same computation for $f^{\inf}$ shows that indeed
$$f^{\sup}(\mu)=\CL(\mu,\tau)=f^{\inf}(\mu).$$
We define this common value to be $\Vert\mu\Vert_X$ and notice that by construction the so defined function $\Vert\cdot\Vert:\CC^+(\Gamma)\to\BR_+$ is continuous. Density of $\BR_+\CS^+(\Gamma)$ ensures that in fact it is the unique continuous extension. Moreover, the so defined stable length extends the stable length on $\BR_+\CS^+(\Gamma)$. To see this note that to compute $\Vert t\gamma\Vert_X$ we can just consider the constant sequence $(t\gamma)$. In particular, the homogeneity of the stable length on $\BR_+\CS^+(\Gamma)$ is inherited by the stable length defined on the whole of $\CC^+(\Gamma)$. Similarly, noting that 
$$\Vert\gamma\Vert_X=\lim_{n\to\infty}\frac 1nd_X(x,\gamma^nx)=\lim_{n\to\infty}\frac 1nd_X(\gamma^{-n}x,x)=\Vert\gamma^{-1}\Vert_X=\Vert\text{flip}(\gamma)\Vert_X,$$
we obtain that the stable length function 
$$\Vert\cdot\Vert_X:\CC^+(\Gamma)\to\BR_+$$
is also flip-invariant. We have proved Theorem \ref{sat-length}.
\end{proof}

\section{}\label{sec-counting}

In this section we prove Theorems \ref{sat1} and \ref{sat2}, and explain how Corollaries \ref{kor-rational}, \ref{kor-riemannian}, and \ref{kor-wordlength} follow. 

\subsection{Currents on surfaces} 
Suppose that $\Sigma$ is a compact hyperbolic surface of genus $g$ and with $r$ boundary components. Let $\Gamma=\pi_1(\Sigma)$ be its fundamental group and $X$ a geodesic metric space on which $\Gamma$ acts discretely and cocompactly by isometries. We identify curves in $\Sigma$, homotopy classes of curves in $\Sigma$, and conjugacy classes in $\Gamma$.  Recall that $\CS(\Sigma)$ denotes the set of all curves in $\Sigma$ and $\CS_{\gamma_0}=\Map(\Sigma)\cdot\gamma_0$ is the set of all curves of type $\gamma_0$. We also denote by $\CC(\Sigma)=\CC(\pi_1(\Sigma))$ the space of currents associated to the hyperbolic group $\pi_1(\Sigma)$. There are several reasons why, in the case of surface groups, the natural thing to consider are currents instead of oriented currents. Recall in any case that we can identify currents with flip-invariant oriented currents. In particular, Theorem \ref{sat-length} still holds if we consider currents instead of oriented currents. 

%As in the general setting, every primitive curve in $\Sigma$ can be considered as a current: namely the Dirac measure centred at the said curve. In this way we can see the set $\CS(\Sigma)$ of all curves in $\Sigma$ as a (dense) subset of the space $\CC(\Sigma)$ of currents on $\Sigma$. In \cite{Bonahon88} Bonahon defines a symmetric map $\iota(\cdot,\cdot): \mathcal{C}(\Sigma)\times\mathcal{C}(\Sigma)\to\mathbb{R}_{+}$ which extends the geometric intersection number of curves on $\Sigma$ to the space of currents. This map, called the \emph{intersection form}, is symmetric, bi-homogenous, and invariant under the action of $\pi_1(\Sigma)$. 

Besides the currents associated to curves $\gamma\in\CS(\Sigma)$, we will be interested in another prominent class of currents. Recall that a measured lamination is a geodesic lamination endowed with a transverse measure of full support. As such, a measured lamination is also a current. In fact, a current $\lambda\in\CC(\Sigma)$ is a measured lamination if and only if $\iota(\lambda,\lambda)=0$, that is
$$\CM\CL=\{\lambda\in\CC(\Sigma)\,\vert\,\iota(\lambda,\lambda)=0\}.$$
Here $\iota(\cdot,\cdot)$ is the intersection form on $\CC(\Sigma)$, the unique continuous bihomogenous extension to $\CC(\Sigma)\times\CC(\Sigma)$ of the geometric intersection number \cite{Bonahon86,Bonahon88}. 

Anyways, being a subset of $\CC(\Sigma)$, the space $\CM\CL=\CM\CL(\Sigma)$ of measured laminations has an induced topology. In fact, Thurston proved that $\CM\CL(\Sigma)$ is homeomorphic to $\BR^{6g-6+2r}$. The space $\CM\CL$ does not only have a natural topology, but also a compatible mapping class group invariant PL-manifold structure. In fact, the PL-manifold $\CM\CL$ is in endowed with a mapping class group invariant symplectic structure and hence with a mapping class group invariant measure in the Lebesgue class. This measure is the so-called {\em Thurston measure} $\mu_{\Thu}$. It is an infinite but locally finite measure, positive on non-empty open sets, and satisfies
$$\mu_{\Thu}(L\cdot U)=L^{6g-6+2r}\cdot\mu_{\Thu}(U)$$ 
for all $U\subset\CM\CL$ and $L>0$. See for example \cite{Casson-Bleiler} for definitions and facts about laminations and measured laminations and \cite{Penner-Harer} for a construction of the Thurston measure $\mu_{\Thu}$ and of  the PL-structure on $\CM\CL$. Note that, considering $\CM\CL$ as a subset of the space of currents, we can consider $\mu_{\Thu}$ as a measure on $\CC(\Sigma)$.

In our setting, the Thurston measure will appear (up to multiplicative error) as the limit of the measures
$$\nu_L^{\gamma_0}=\frac 1{L^{6g-6+2r}}\sum_{\gamma\in\CS_{\gamma_0}}\delta_{\frac 1L\gamma}$$
when $L\to\infty$. Here $\delta_x$ stands for the Dirac measure centered at $x$. To see what $\mu_{\Thu}$ actually has to do with the measures $\nu_L=\nu_L^{\gamma_0}$ recall first that by \cite[Proposition 4.1]{VJ} we have that the family $\nu_L$ is precompact and that any limit is a multiple of the Thurston measure. Suppose then that $\nu_{L_n}$ converges to $C\cdot\mu_{\Thu}$. We claim that $C$ does not depend on the chosen sequence. From \cite[Proposition 4.3]{VJ} we get that the convergence $\nu_{L_n}\to C\cdot\mu_{\Thu}$ implies that for any filling current $\lambda$ we have
$$\lim_{n\to\infty}\frac{\vert\{\gamma\in\CS_{\gamma_0}\,\vert\,\iota(\lambda,\gamma)\le L_n\}\vert}{L_n^{6g-6+2r}}=C\cdot\mu_{\Thu}(\{\alpha\in\CM\CL\,\vert\,\iota(\lambda,\alpha)\le 1\})$$
where $\iota(\cdot,\cdot)$ is still the intersection form on $\CC(\Sigma)$. If we apply this fact to the Liouville current $\lambda_X$ of a hyperbolic surface we get, writing $\ell_X$ for the hyperbolic length, that
\begin{equation}\label{eq-hyp-lim}
\lim_{n\to\infty}\frac{\vert\{\gamma\in\CS_{\gamma_0}\,\vert\,\ell_X(\gamma)\le L_n\}\vert}{L_n^{6g-6+2r}}=C\cdot\mu_{\Thu}(\{\alpha\in\CM\CL\,\vert\,\ell_X(\alpha)\le 1\}).
\end{equation}
It is however a theorem of Mirzakhani \cite{Maryam2} that the limit on the left is independent of the sequence $L_n\to\infty$. This implies that the quantity on the right, and in particular $C$, is also independent of the chosen sequence $(L_n)$. We have proved:

\begin{sat}\label{limit measures}
For every curve $\gamma_0\in\CS(\Sigma)$ there is a constant $C_{\gamma_0}>0$ with $\lim_{L\to\infty}\nu_L^{\gamma_0}=C_{\gamma_0}\cdot\mu_{\Thu}.$\qed
\end{sat}

\begin{bem}
The argument we just gave to prove Theorem \ref{limit measures} appeared in \cite{VJ}, although the result was as such never formally stated in the said paper. A more general version of Theorem \ref{limit measures}---which applies if we replace $\gamma_0$ by an arbitrary current, but at this point only to closed surfaces---was recently given by Rafi and the third author of this paper \cite{Kasra-Juan}. 
\end{bem}

Let us be a bit more careful with the constants. In \cite{Maryam2}, Mirzikhani proves that there is a rational constant $n_\gamma\in\BQ$ such that the limit \eqref{eq-hyp-lim} is equal to the product of 
$$n_{\gamma_0}\cdot\mu_{\Thu}(\{\alpha\in\CM\CL\,\vert\,\ell_X(\alpha)\le 1\})\cdot b_{\Sigma}^{-1}$$
where $b_{\Sigma}$ is the integral of the function $X\mapsto\mu_{\Thu}(\{\alpha\in\CM\CL\,\vert\,\ell_X(\alpha)\le 1\})$ over moduli space endowed with the Weil-Peterson metric. It follows that the constant $C_{\gamma_0}$ in Theorem \ref{limit measures} can be expressed as $C_{\gamma_0}=n_{\gamma_0}\cdot b_{\Sigma}^{-1}.$

\subsection{Proofs of Theorems \ref{sat1} and \ref{sat2}}

We are now ready to prove Theorem \ref{sat2}. We recall its statement: 

\begin{named}{Theorem \ref{sat2}}
Let $\Sigma$ be a compact surface with genus $g$ and $r$ boundary components and assume that $3g+r > 3$. Let also $\pi_1(\Sigma)\actson X$ be a discrete and cocompact isometric action on a geodesic metric space $X$, $\gamma_0$ an essential curve in $\Sigma$, and $\CS_{\gamma_0}=\Map(\Sigma)\cdot\gamma_0$ the set of all curves of type $\gamma_0$. Then the limit
$$\lim_{L\to\infty}\frac 1{L^{6g-6+2r}}\vert\{\gamma\in\CS_{\gamma_0}\text{ with }\Vert\gamma\Vert_X\le L\}\vert$$
exists and is positive. 
\end{named}

\begin{proof}
By Theorem \ref{sat-length} the stable length extends to a continuous, homogenous function $\Vert\cdot\Vert_X: \CC(\Sigma)\to\mathbb{R}_+$. From its homogeneity we have
$$\vert\{\gamma\in\CS_{\gamma_0}\vert\Vert\gamma\Vert_X\le L\}\vert=\left\vert\left\{\gamma\in\CS_{\gamma_0}\middle\vert\left\Vert\frac 1L\gamma\right\Vert_X\le 1\right\}\right\vert.$$
Recalling the measures 
$$\nu_L^{\gamma_0}=\frac 1{L^{6g-6+2r}}\sum_{\gamma\in\CS_{\gamma_0}}\delta_{\frac 1L\gamma}$$
on the space of currents $\CC(\Sigma)$ that we considered above, we can thus write
$$\frac 1{L^{6g-6+2r}}\vert\{\gamma\in\CS_{\gamma_0}\,\vert\,\Vert\gamma\Vert_X\le L\}\vert=\nu_L^{\gamma_0}\left(\left\{\lambda\in\CC(\Sigma) \,\vert\,\Vert\lambda\Vert_X\le 1\right\}\right).$$
The set $\left\{\lambda\in\CC(\Sigma) \,\vert\,\Vert\lambda\Vert_X\le 1\right\}$ is closed. Moreover, the homogeneity of $\Vert~\cdot~\Vert_X$ and the scaling behavior of the Thurston measure $\mu_{\Thu}$, imply that 
$$\mu_{\Thu}(\D \left(\left\{\lambda\in\CC(\Sigma)\, \vert\,\Vert\lambda\Vert_X\le 1\right\}\right))=0.$$
Since the measures $\nu_L^{\gamma_0}$ converge to $C_{\gamma_0}\cdot\mu_{\Thu}$ by Theorem \ref{limit measures}, we get thus that the $\nu_L^{\gamma_0}$ measures of $\left\{\lambda\in\CC \,\vert\,\Vert\lambda\Vert_X\le 1\right\}$ converge to its $C_{\gamma_0}\cdot\mu_{\Thu}$ measure. Altogether we get that
$$\lim_{L\to\infty}\frac 1{L^{6g-6+2r}}\vert\{\gamma\in\CS_{\gamma_0}\,\vert\,\Vert\gamma\Vert_X\le L\}\vert=C_{\gamma_0}\cdot\mu_{\Thu}\left(\left\{\lambda\in\CC(\Sigma)\, \vert\,\Vert\lambda\Vert_X\le 1\right\}\right).$$
We have proved the theorem.
\end{proof}

Recall that by Lemma \ref{lem-compare lengths} we have, for all $\epsilon$,
$$\ell_X(\gamma)\ge\Vert\gamma\Vert_X\ge(1-\epsilon)\cdot\ell_X(\gamma)$$
for every curve $\gamma$ with $\ell_X(\gamma)$ large enough. Using this together with Theorem \ref{sat2} above, Theorem \ref{sat1} is almost immediate: 

\begin{named}{Theorem \ref{sat1}}
Let $\Sigma$ be a compact surface with genus $g$ and $r$ boundary components and assume that $3g+r> 3$. Let also $\pi_1(\Sigma)\actson X$ be a discrete and cocompact isometric action on a geodesic metric space $X$, $\gamma_0$ an essential curve in $\Sigma$, and $\CS_{\gamma_0}=\Map(\Sigma)\cdot\gamma_0$ the set of all curves of type $\gamma_0$. Then the limit
$$\lim_{L\to\infty}\frac 1{L^{6g-6+2r}}\vert\{\gamma\in\CS_{\gamma_0}\,\vert\,\ell_X(\gamma)\le L\}\vert$$
exists and is positive. 
\end{named}

\begin{proof}
We will prove that the limit agrees with the limit in Theorem \ref{sat1}. In fact there is nothing to be proved because Lemma \ref{lem-compare lengths} implies that
\begin{align*}
\liminf_{L\to\infty}\frac{\vert\{\gamma\in\CS_{\gamma_0}\vert\ell_X(\gamma)\le L\}\vert}{L^{6g-6+2r}}&\ge\lim_{L\to\infty}\frac{\vert\{\gamma\in\CS_{\gamma_0}\vert\Vert\gamma\Vert_X\le L\}\vert}{L^{6g-6+2r}}\\
&\ge (1-\epsilon)^{6g-g+2r}\limsup_{L\to\infty}\frac{\vert\{\gamma\in\CS_{\gamma_0}\vert\ell_X(\gamma)\le L\}\vert}{L^{6g-6+2r}}
\end{align*}
for all $\epsilon>0$. The claim follows.
\end{proof}

Note that in the proof of Theorem \ref{sat1} we have actually proved that the limit therein agrees with the limit in Theorem \ref{sat2}. In particular, from the proof of Theorem \ref{sat2} we know that we actually have
$$\lim_{L\to\infty}\frac 1{L^{6g-6+2r}}\vert\{\gamma\in\CS_{\gamma_0}\,\vert\,\ell_X(\gamma)\le L\}\vert=C_{\gamma_0}\cdot\mu_{\Thu}\left(\left\{\lambda\in\CC(\Sigma)\, \vert\,\Vert\lambda\Vert_X\le 1\right\}\right).$$
This means that if $\eta$ and $\eta'$ are curves then we have
$$\lim_{L\to\infty}\frac{\vert\{\gamma\in\CS_{\eta}\,\vert\,\ell_X(\gamma)\le L\}\vert}{\vert\{\gamma\in\CS_{\eta'}\,\vert\,\ell_X(\gamma)\le L\}\vert}=\frac{C_\eta}{C_{\eta'}}$$
Now, as mentioned after the proof of Theorem \ref{limit measures} we have that $C_{\eta}=n_{\eta}\cdot b_{\Sigma}^{-1}$ and correspondingly $C_{\eta'}=n_{\eta'}\cdot b_{\Sigma}^{-1}$. Taking all of this together we obtain Corollary \ref{kor-rational}:

\begin{named}{Corollary \ref{kor-rational}}
Let $\Sigma$ be a compact surface with genus $g$ and $r$ boundary components and assume that $3g+r> 3$. For every curve $\eta$ in $\Sigma$ there is $n_{\eta}\in\BQ$ such that
$$\lim_{L\to\infty}\frac{\vert\{\gamma\in\CS_{\eta}\,\vert\,\ell_X(\gamma)\le L\}\vert}{\vert\{\gamma\in\CS_{\eta'}\,\vert\,\ell_X(\gamma)\le L\}\vert}=\frac{n_\eta}{n_{\eta'}}$$
for any two essential curves $\eta,\eta'$ in $\Sigma$ and any discrete and cocompact action $\pi_1(\Sigma)\actson X$.\qed
\end{named}

We now discuss other consequences of Theorem \ref{sat1}. Note that, as already mentioned in the introduction, Corollary \ref{kor-wordlength} follows immediately from Theorem \ref{sat1} by setting $X$ to be the Cayley graph $C(\pi_1(\Sigma),S)$. Also, if $\Sigma$ is closed and $\rho$ a Riemannian metric on $\Sigma$ then it follows directly from Theorem \ref{sat1} that the limit
$$\lim_{L\to\infty}\frac 1{L^{6g-6}}\vert\{\gamma\in\CS_{\gamma_0}\,\vert\,\ell_\rho(\gamma)\le L\}\vert$$
exists and is positive. Corollary \ref{kor-riemannian} asserts that this is also the case if the surface $\Sigma$ has boundary and if $\rho$ is a complete metric on its interior. This is what we prove next.

\subsection{Proof of Corollary \ref{kor-riemannian}}
Suppose that $\Sigma$ is as in the statement of Theorem \ref{sat1}, denote by $\Sigma^0=\Sigma\setminus\D\Sigma$ its interior and let $\rho$ be a complete Riemannian metric on $\Sigma^0$. The basic difficulty when trying to derive Corollary \ref{kor-riemannian} from Theorem \ref{sat1} is that the action of $\pi_1(\Sigma)=\pi_1(\Sigma^0)$ on the universal cover of $(\Sigma^0,\rho)$ is not cocompact. Our aim is to replace this action by some cocompact action. What we really need to prove is that all the relevant geodesics are contained in a fixed compact set:

\begin{lem}\label{lem-riemann}
For any essential curve $\gamma_0$ in $\Sigma$, there exists a compact subsurface $(\Sigma_K,\rho)$ of $(\Sigma^0,\rho)$ such that any minimal length curve freely homotopic to some element in $\CS_{\gamma_0}$ is entirely contained in $(\Sigma_K,\rho)$.
\end{lem}

Note that for a complete hyperbolic surface the action always happens in the convex core of the surface. Still in the hyperbolic world, the convex core fails to be compact only in the presence of cusps and it is well-known that closed geodesics that go far up a cusp must wind many times around the cusp. This is turn creates self-intersections but the self-intersection number only depends on the type of the curve and not on its geodesic realization. This proves Lemma \ref{lem-riemann} for hyperbolic metrics. We  apply basically the same argument in the general case, using the fact that curves of minimal length on Riemannian surfaces realize minimal intersection number---this is a result of Freedman, Hass and Scott \cite{FreedmanHassScott}. 

\begin{proof}
We begin by fixing a smooth compact subsurface, say $(\Sigma',\rho)$, of $(\Sigma^0,\rho)$ homeomorphic to $\Sigma$. It can be obtained in such a way that $(\Sigma^0,\rho) \setminus (\Sigma',\rho)$ is a collection of $r$ cylinders. Denote by $\beta_1, \hdots,\beta_r$ the boundary curves of $(\Sigma',\rho)$. 

\begin{figure}[h]
%\ShowGrid
{
\leavevmode \SetLabels
\L(.48*.35) $\beta_i$\\%
\L(.49*.05) $\Sigma'$\\%
%\L(.40*.97) $\Psi$\\
\endSetLabels
\begin{center}
\AffixLabels{\centerline{\epsfig{file =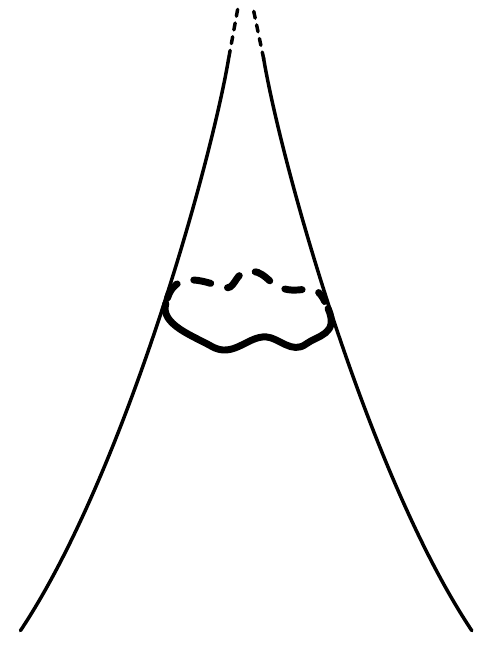,width=2.0cm,angle=0}}}
\vspace{-24pt}
\end{center}
}
\caption{The subsurface $\Sigma'$} \label{fig:Cusp1}
\end{figure}

For each $\gamma \in \CS_{\gamma_0}$ we choose a single minimal length closed geodesic, which we denote $\gamma_\rho$, in its homotopy class. As noted above, $\gamma_\rho$ enjoys a minimal intersection property \cite{FreedmanHassScott}, namely $\iota(\gamma_\rho,\gamma_\rho)$ is minimal among all curves in the same free homotopy class. All elements in the mapping class group orbit $\CS_{\gamma_0}$ have the same self-intersection number and thus so do any of their minimal length representatives. We denote this number by $I$ below.

If a given $\gamma_\rho$ is not entirely contained in $(\Sigma',\rho)$, then the restriction of $\gamma_\rho$ to each cylinder of $(\Sigma^0,\rho) \setminus (\Sigma',\rho)$ is a collection of geodesic arcs. Consider the cylinder of boundary $\beta_i$ and a geodesic subarc $a$ of $\gamma_\rho$ in the cylinder with both endpoints on $\beta_i$. By minimality of the length of $\gamma_\rho$, $a$ is of minimal length among all possible homotopic arcs with the same endpoints. Furthermore, the homotopy type of arc is entirely determined by the number of times $a$ wraps around the cylinder (and in which direction). In particular, $a$ is homotopic (with fixed endpoints) to any arc which is entirely contained in $\beta_i$ (seen as a set of points) which wraps the same number of times around the boundary.

\begin{figure}[h]
%\ShowGrid
{
\leavevmode \SetLabels
\L(.50*.0) $\beta_i$\\%
\L(.52*.34) $a$\\%
%\L(.40*.97) $\Psi$\\
\endSetLabels
\begin{center}
\AffixLabels{\centerline{\epsfig{file =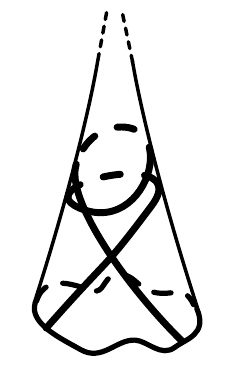,width=2.0cm,angle=0}}}
\vspace{-24pt}
\end{center}
}
\caption{An arc $a$ that wraps twice around the boundary} \label{fig:Cusp2}
\end{figure}

Further note that $a$ cannot wrap more than $I$ times around the cylinder, otherwise the self-intersection number of $a$ and thus of $\gamma_\rho$, would be too large. In particular this implies that $\ell(a) \leq (I+2) \ell(\beta_i)$ and is thus strictly contained in the $\frac{I+2}{2} \ell(\beta_i)$ neighborhood of $\beta_i$. If we set $C=\max_{i}\left\{\frac{I+2}{2} \ell(\beta_i)\right\}$ then $\gamma_\rho$ is entirely contained in the $C$ neighborhood of  $\Sigma'$. This is the desired compact set.
\end{proof}

We are now ready to prove Corollary \ref{kor-riemannian}:

\begin{named}{Corollary \ref{kor-riemannian}}
With $\Sigma$ as in Theorem \ref{sat1}, let $\rho$ be a complete Riemannian metric on $\Sigma_0=\Sigma\setminus\D\Sigma$. Then for every essential curve $\gamma_0$ the limit
$$\lim_{L\to\infty}\frac 1{L^{6g+2r-6}}\vert\{\gamma\in\CS_{\gamma_0}\,\vert\,\ell_\rho(\gamma)\le L\}\vert$$
exists and is positive. Here $\ell_\rho(\gamma)$ is minimum of the lengths with respect to $\rho$ over all curves freely homotopic to $\gamma$.
\end{named}
\begin{proof}
Let $\Sigma_K$ be as provided by Lemma \ref{lem-riemann} and note that we can assume without loss of generality that the inclusion $\Sigma_K\hookrightarrow \Sigma$ is a homotopy equivalence. Let $X$ be the universal cover of $(\Sigma^0,\rho)$ and $X'$ that of $(\Sigma_K,\rho)$ and note that there is an embedding $X'\to X$ which is equivariant under the deck-transformation actions of $\pi_1(\Sigma)=\pi_1(\Sigma^0)=\pi_1(\Sigma_K)$. 

The action of $\pi_1(\Sigma)$ on $X'$ is cocompact. In particular, it follows from Theorem \ref{sat1} that the limit
$$\lim_{L\to\infty}\frac 1{L^{6g-6+2r}}\vert\{\gamma\in\CS_{\gamma_0}\,\vert\,\ell_{X'}(\gamma)\le L\}\vert$$
exists and is positive. On the other hand, since $\Sigma_K$ contains every minimizing geodesic in $(\Sigma,\rho)$ homotopic to any curve $\gamma$ of type $\gamma_0$, we get that
$$\ell_\rho(\gamma)=\ell_X(\gamma)=\ell_{X'}(\gamma).$$
for any $\gamma\in\CS_{\gamma_0}$. Altogether, it follows that the limit
$$\lim_{L\to\infty}\frac 1{L^{6g-6+2r}}\vert\{\gamma\in\CS_{\gamma_0}\,\vert\,\ell_\rho(\gamma)\le L\}\vert$$
also exists and is positive, as we needed to prove.
\end{proof}

\end{document}